\pdfoutput=1
\documentclass[10pt, reqno]{amsart}

\title[Capillary surfaces arising in singular perturbation problems]{Capillary surfaces arising in singular perturbation problems}
\author{Aram L. Karakhanyan}
\address{School of Mathematics, The University of Edinburgh, Peter Tait Guthrie Road, EH9 3FD, Edinburgh, UK}
\email{aram6k@gmail.com}

\usepackage{color,graphicx,enumerate,wrapfig,amssymb,setspace}

\usepackage[pagebackref,linktocpage=true,colorlinks=true,linkcolor=Blue,citecolor=magenta,urlcolor=RoyalBlue]{hyperref} 

\usepackage[usenames,dvipsnames,svgnames,table]{xcolor}

\usepackage[alphabetic,msc-links,abbrev]{amsrefs} 

\usepackage{enumitem}

\usepackage{verbatim}

\usepackage{skak} 

\usepackage[mathscr]{euscript} 

\usepackage{esint}

\usepackage{graphicx, geometry}
\renewenvironment{proof}[1][\proofname ]{{\noindent \bfseries #1. }}{\qed \bigskip }

\newcommand{\R}{{\mathbb R}}

\newcommand\po[1]{ \{ {#1} > 0 \}}
\newcommand{\m}[1]{\mathcal{#1} }

\newcommand{\e}{\varepsilon}

\newcommand{\supp}{\operatorname{supp}}
\newcommand{\D}{\nabla}

\renewcommand\H{\mathcal H}
\newcommand\B{\mathcal B}
\newcommand\p{\partial}
\newcommand\fb[1]{\p \{ {#1} > 0 \}}
\newcommand\fbr[1]{\p_{\rm red} \{ {#1} > 0 \}}

\def\Om{\Omega}
\def\na{\nabla}

\newcommand\I[1]{\chi_{\{#1>0\}}}  
\def\S{\mathbb S}
\def\lap{\triangle}
\newcommand\M{\mathscr M}
\newcommand\X{\mathscr X}
\newcommand\C{\mathscr C}
\newcommand\vol{\text{vol}}

\newtheorem{theorem}{Theorem}[section]

\newtheorem{cor}[theorem]{Corollary}

\newtheorem{defn}{Definition}[section]

\newtheorem{lem}[theorem]{Lemma}

\newtheorem{prop}[theorem]{Proposition}
\newtheorem{remark}[theorem]{Remark}

\newtheoremstyle{named}{}{}{\itshape}{}{\bfseries}{.}{.5em}{\thmnote{#3 }#1}
\theoremstyle{named}

\newtheorem*{AAA}{Theorem A}

\newtheorem*{BBB}{Theorem B}
\newtheorem*{CCC}{Theorem C}

\usepackage{tikz}
\usetikzlibrary{arrows,shapes,positioning}
\usetikzlibrary{decorations.markings}
\tikzstyle arrowstyle=[scale=1]
\tikzstyle directed=[postaction={decorate,decoration={markings,
    mark=at position .65 with {\arrow[arrowstyle]{stealth}}}}]
\tikzstyle reverse directed=[postaction={decorate,decoration={markings,
    mark=at position .65 with {\arrowreversed[arrowstyle]{stealth};}}}]

\numberwithin{equation}{section}

\thanks{2000 Mathematics Subject Classification. Primary 49Q05, 35R35, 35B25. 
\\ Keywords: Singular perturbation problem, free boundary regularity, capillary surfaces, global solutions.}

\begin{document}
\maketitle

\baselineskip=14pt    
\begin{abstract}
In this paper we prove  some Bernstein type theorems for a class of  stationary points of the
Alt-Caffarelli functional in $\R^2$ and $\R^3$arising as  limits of the singular perturbation problem
\begin{equation}\label{abs-pde}
\left\{
\begin{array}{lll}
\lap u_\e(x)=\beta_\e(u_\e)& \mbox{in}\quad B_1, \\
|u_\e|\le 1&\mbox{in} \quad  B_1, \\
\end{array}
\right.
\end{equation}
 in the unit ball $B_1$ as $\e\to 0$. Here $\beta_\e(t)=\frac1\e\beta(t/\e)\ge 0, \beta\in C_0^\infty[0,1], \int_0^1\beta(t)dt=M>0$,
is an approximation of the Dirac measure and $\e>0$.
The limit functions $u=\lim\limits_{\e_j \to 0}u_{\e_j}$ of 
uniformly converging sequences $\{u_{\e_j}\}$ solve a Bernoulli type free boundary problem in some weak sense.
Our approach has two novelties:
First we develop a hybrid method for stratification of the free boundary 
$\fb{u_{0}}$
of blow-up solutions which combines some ideas and techniques 
of viscosity and variational theory.
An important tool we use is a new monotonicity formula for the solutions $u_\e$ based on a computation of J. Spruck.
It  implies that any blow-up  $u_0$ of $u$ either vanishes identically or is
homogeneous function of degree one, that is $u_0=rg(\sigma), \sigma\in \S^{N-1}$ in spherical coordinates $(r, \theta)$.
In particular, this implies that in two dimensions the singular set is empty at the non-degenerate points, and in
three dimensions the singular set of $u_0$ is at most a singleton.   
Second, we show that the spherical part $g$ is the support function (in Minkowski's sense)
of some capillary surface contained in the sphere of radius $\sqrt{2M}$.
In particular, we show that $\nabla u_0:\S^2\to \R^3$ is an almost conformal and minimal immersion and
the singular Alt-Caffarelli example corresponds to a piece of catenoid
which  is a unique ring-type
stationary minimal surface determined by the support function $g$.
\end{abstract}

\tableofcontents

\section{Introduction}
In this paper we study the singular perturbation problem
\begin{equation}\label{pde-0}\tag{$\mathcal P_\e$}
\left\{
\begin{array}{lll}
\lap u_\e(x)=\beta_\e(u_\e) & \mbox{in}\quad B_1, \\
|u_\e|\le 1 & \mbox{in}\quad B_1, 
\end{array}
\right.
\end{equation}
where $\e>0$ is a small parameter and
\begin{eqnarray}\label{beta}
\left\{\begin{array}{lll}
\beta_\e(t)=\frac1\e\beta\left(\frac{t}\e\right), \\
\beta(t)\ge 0, \quad \supp\beta\subset [0, 1], \\
\int_0^1\beta(t)dt= M>0,
\end{array}
\right.
\end{eqnarray}
is an approximation of the Dirac measure, $B_1\subset \R^N$ is the unit ball centered  at the origin.
It is well known that \eqref{pde-0}
models propagation
of equidiffusional premixed  flames with high activation of energy \cite{C-95}.
Heuristically, the limit $u_0=\lim\limits_{\e_j\to 0}u_{\e_j}$ (for a suitable sequence $\e_j\to 0$) solves a Bernoulli type free boundary problem
with the following free boundary condition
$$|\na u^+|^2-|\na u^-|^2=2M.$$
If the functions $\{u_\e\}$ are also  minimizers of
\begin{equation}
J_\e[u_\e]=\int_{\Omega}\frac{|\na u_\e|^2}2+\B(u_\e/\e), \quad \B(t)=\int_0^t\beta(s)ds, 
\end{equation}
then the limits of $\{u_\e\}$ inherit the generic features of minimizers (e.g. non-degeneracy, rectifiability of $\fb u$, etc.).
Consequently, the limits of uniformly converging  sequences
$\{u_{\e_j}\}$ as $\e_j\to 0$ are  minimizers of
the Alt-Caffarelli functional $J[u]=\int_{B_1}|\na u|^2+2M\chi_{\{u>0\}}$.
It is known that the singular set of minimizers is empty in dimensions $2, 3$ and  $4$, see  \cite{AC}, \cite{CJK}, \cite{JS}.
However, if $u_\e$ is not a minimizer then the analysis of 
the limits $u$ presents a more delicate problem. 
The main difficulty in carrying out such analysis is that 
the free boundary may contain degenerate points \cite{Weiss}.

\smallskip 

This paper is devoted to the study of the blow-ups of the 
limits of the singular perturbation problem \eqref{pde-0} and establishes a new and
direct connection with  {\bf minimal surfaces}. In particular, we
show that every blow-up of a limit function $u=\lim_{\e_j\to 0}u_{\e_j}$ in $\R^3$
(for an appropriate sequence $\e_j$)
defines an almost  conformal and minimal immersion which is perpendicular to the sphere
of radius $\sqrt{2M}$ where $M=\int_0^1\beta(t)dt$. In other words,  one obtains a capillary surface  inside the
sphere
of radius $\sqrt{2M}$. 

\medskip
Our  first result  is
\begin{AAA}
Let $u_{\e_j}\to u$ locally uniformly in $B_1$ for some subsequence $\e_j$,
then any blow-up of $u$ at free boundary point $x_0\in \fb u$ is either
identically zero or homogeneous function of degree one.
In particular, if  $N=2$ and $u$ is not degenerate at  $x_0\in \fb u$ then every blow-up of $u$ at $x_0$  must be one of the following functions (after some rotation of coordinates):
\begin{itemize}
\item[(1)] $\sqrt{2M}x_1^+$, \hbox{half plane solution provided that there is a measure theoretic normal at}\ $x_0$,
\item[(2)]{wedge} $\alpha |x_1|, 0<\alpha\le\sqrt{2M}$,
\item[(3)]{two plane solution} $\alpha x_1^+-\beta x_1^-, \alpha^2-\beta^2=2M, \alpha, \beta>0$.
\end{itemize}

\end{AAA}

\medskip
In order to prove Theorem A we will introduce a monotone quantity 
based on a computation of Joel Spruck  \cite{Spruck}.
From Theorem A it follows that in $\R^2$ the blow-up limits at  non-degenerate free boundary points
can be explicitly computed.
It is worthwhile to note that the minimizers of
\begin{equation}
J[u]=\int_{B_1}|\na u|^2+ 2M\I {u}
\end{equation}
are non-degenerate, i.e. for each subdomain $\Omega'\subset\subset B_1$ there is a constant $c_0>0$
depending on $\mbox{dist}(\partial B_1,\partial\Omega' )$, $N$, $M$,  such that
\begin{equation}\label{nondeg}
\sup_{B_r(x_0)}u^+\ge c_0r, \qquad \forall x_0\in \fb u\cap \Omega', \quad B_r(x_0)\subset B_1.
\end{equation}
However, if $u_\e$ is {\bf any} solution of \eqref{pde-0} then non-degeneracy may not be true.
There is a sufficient condition \cite{CLW-uniform} Theorem 6.3 that implies \eqref{nondeg}.

Some well-known examples demonstrate rather
strikingly that for the stationary case there are wedge-like global solutions
for which the measure theoretic  boundary of $\po u$ is empty. This is impossible for minimizers. In fact,  the zero set  of a minimizer 
has uniformly positive Lebesgue density.
In this respect Theorem A only states that if $u$ is non-degenerate at $x_0$ then the blow-up is a nontrivial cone.

The existence of wedge solutions (see Remark 5.1 \cite{CLW-uniform}) suggests that some further assumptions are needed to
formulate the free boundary condition. 
For instance, one may 
assume that the upper Lebesgue density at $x\in \fb u$ 
satisfies $\Theta^{*}(x, \{u>0\})<1$,
i.e. the upper density measure is not covering the full ball.
%
We emphasize that for some  solutions the topological and measure theoretic boundaries may not coincide. 
Our next result addresses the degeneracy and wedge-formation in $\R^3$
of blow-ups at  free boundary points.

\begin{BBB}
Suppose $N=3$. Let  $u\ge 0$ be a limit of some 
uniformly converging sequence $\{u_{\e_j}\}$ solving \eqref{pde-0}
such that $u$ is non-degenerate at $y_0\in \fb u$.
Let $u_0$ be a blow-up of $u$ at $y_0$. If $\mathfrak C$ is a component of
$\fb{u_0}$ such that the measure theoretic boundary of $\po{u_0}$ in  $\mathfrak C$ is non-empty
then
\begin{itemize}
\item[(1)] all points of $\mathfrak C$ are non-degenerate,
\item[(2)]  $\mathfrak C$ is a subset of the  measure theoretic boundary of $\po{u_0}$,
\item[(3)]  $\mathfrak C\setminus\{0\}$ is smooth.
\end{itemize}
In particular in $\R^3$ the singular set of $\fb{u_0}$ is atmost a singleton.
\end{BBB}

Theorem B implies that the reduced boundary propagates instantaneously in the
components of $\fb{u_0}$.
Our last result sheds some new light on the characterization 
of the blow-ups as minimal surfaces 
inside spheres with contact angle $\pi/2$.

\begin{CCC}
Let $u_0$ be as in Theorem B and $u_0=rg(\sigma), \sigma\in \S^2$ in spherical coordinates. Then
the parametrization $\X(\sigma)=\sigma g(\sigma)+\nabla_{\S^2}g(\sigma)$ defines
an almost conformal and minimal immersion. If $\po g$ is homeomorphic to a disk
then $u_0$ is a half-plane solution $\sqrt{2M}x_1^+$.
If $\po g$ is homeomorphic to a ring then
the only singular cone is the Alt-Caffarelli catenoid.

\end{CCC}

Observe that $\lap u_0=0$ implies that the spherical part $g$ satisfies the follwoing equation on the sphere
\[\lap_{\S^{N-1}}g+(N-1)g=0,\]
where $\lap_{\S^{N-1}}$ is the Laplace-Beltrami operator.
If we regard $g$ as the support function of some embedded hypersurface $\M$ then
the matrix $[\nabla_{ij}g+\delta_{ij}g]^{-1}$ gives the Weingarten mapping and its eigenvalues are the
principal curvatures $k_1, \dots, k_{N-1}$ of $\M$. 
If $N=3$ then we have that
\[0=\lap_{\S^2}g+2g={\rm{trace}}[\nabla_{ij}g+\delta_{ij}g]=\frac1{k_1}+\frac1{k_2}=\frac{k_1+k_2}{k_1k_2}\]
implying that the mean curvature is zero at the points where the Gauss curvature $k_1k_2$ does not vanish.
This is how the minimal surfaces enter into the game. One of the main obstacles is to show
that the surface parametrized by $\X(\sigma)=\nabla u_0(\sigma)$  is embedded. Then the
classification for the disk-type domains $\po{g}$ follows from a result of Nitsche \cite{Nitsche}. To prove
the last statement of Theorem C we will use the moving plane method.
It is worthwhile to point out that the results of this paper can be extended to other classes 
of  stationary points. For instance, the weak solutions introduced in \cite{AC} 
can be analyzed in similar way provided that the zero set has uniformly positive 
Lebesgue density at free boundary points in order to guarantee that the class of 
weak solution is closed with respect to  blow-ups, see example 5.8 in \cite{AC}. 

\subsection*{Related works}
In \cite{HPP} F. H\'elein, L. Hauswirth and F. Pacard have considered the following 
overdetermined problem
\begin{equation}\label{except}
\left\{
\arraycolsep=0.5pt
\begin{array}{lll}
\lap u(x)=0 \ \ & \mbox{in}\quad \Omega, \\
u>0\ \ &\mbox{in}\quad \Omega,\\
u(x)=0, \ |\na u|=1 \ \ & \mbox{on}\quad \p\Omega, \\
\end{array}
\right.
\end{equation}
where $\Om$ is a smooth domain and the boundary conditions are satisfied in the classical sense. 
A domain $\Om$ admitting a solution $u$ to \eqref{except} is called exceptional. Note that every nonnegative smooth solution of the limiting singular perturbation problem solves \eqref{except} with $M=\frac12$.   In \cite{HPP} the authors have  
constructed a number of examples of exceptional domains  and proposed to classify them. 
In particular, they proved that if $\Omega\subset \R^2$ is conformal to half-plane such that $u$ is strictly monotone in one fixed direction then 
$\Omega$ is a half-space, \cite[Proposition 6.1]{HPP}.
However the general problem remained open. 

\medskip 


Later M. Traizet showed that the smoothness assumption can be relaxed, namely 
if $\Om\subset \R^2$ has $C^0$ boundary and the boundary conditions are still satisfied in the classical sense then 
$\Om$ is real analytic \cite[Proposition 1]{T}. 
Under various topological conditions on the two dimensional domain $\Om\subset \R^2$ (such as finite connectivity and  periodicity) M. Traizet classified 
the possible exceptional domains. One of his remarkable results is that from $\Omega$ one can construct a {\bf \it complete} minimal surface using the Weierstrass representation formula
\cite[Theorem 9]{T}.
Another classification result in $\R^2$, under
stronger topological hypotheses than in \cite{T}, is given by D. Khavinson, E. Lundberg and R. Teodorescu \cite{Khavinson}.
Moreover, their results in simply connected case  are stronger because unlike M. Traizet they do not assume the finite connectivity (i.e. $\p \Om$ has finite number of components). As opposed to these results (1) we do not assume any regularity of the  free boundary (which plays the role of $\p \Om$ in \eqref{except}), (2) the Neumann condition is not satisfied in the classical sense, (3) the minimal surface we construct in Theorem C is {\it not complete} and it is a capillary surface inside sphere, and (4) our techniques do not impose any restriction on the dimension.  
Note that, in \cite{HPP} the authors  suggested to study more general classes of exceptional domains:
if  $(M, g)$ is an $m$-dimensional Riemannian manifold  admitting  a harmonic function 
with zero Dirichlet  and constant Neumann boundary data then $M$ is called exceptional
and $u$ a roof function. In this context Theorem C provides a way of constructing roof function on the 
sphere from the blow-ups of stationary points of the Alt-Caffarelli functional.

One may consider higher order critical points as well, such as mountain passes (which are, in fact, minimizers over some subspace of admissible functions) for which one has non-degeneracy and nontrivial Lebesgue density properties \cite[Propositions 1.7-5.1]{JK}. 
Observe that neither of these properties is available for our solutions as Theorem 6.3 and Remark 5.1 in \cite{CLW-uniform} indicate,  and in the present work we do not impose any additional 
assumptions on our stationary points of this kind.  


\medskip

It seems that the only  result in high dimensions that appears in \cite{HPP}, \cite{Khavinson} and \cite{T}  states that if the complement of $\Omega$ is connected and has $C^{2,\alpha}$ boundary, then $\Omega$ is the exterior of a ball \cite[Theorem 7.1]{Khavinson}. The restriction $\Omega\subset \R^2$ is because the authors have mainly  used the techniques from complex analysis. 
Our approach does not have this restriction  since our main tool is the representation of the solution in terms of the Minkowski support function.
We remark that using our method in high dimensions we can  construct a surface $\M$ inside the sphere of 
radius $\sqrt{2M}$ such that the sum of its principal radii of curvature is zero,  and $\M$ is transversal to the sphere. 

Finally,  we point out that our approach may lead to a  new characterization  of global minimizers in $\R^3$ \cite{CJK}. Indeed,   Theorem 6 from \cite{RV} implies that 
the capillary surface $\M$ in Theorem C associated with the blow-up limit must be totally geodesic (i.e. the second fundamental form is identically zero). Consequently, the blow-up must be the half-plane solution.

\smallskip

The paper is organized as follows:  In Section 2 we  set up  some basic 
notation which will be used throughout the paper.
Section 3 is devoted to the study of a new monotone quantity $s(x_0,u, r)$. This interesting quantity is derived from a 
computation of J. Spruck \cite{Spruck}. 
Among other things, properties of $s$ 
imply that every blow-up of $u$ is either homogeneous function of 
degree one or identically zero.
Section 4 contains the proof of Theorem A.
In Section 5 
we develop a new method of stratification of the free boundary points and prove Theorem B.  
Section 6 contains the proof of Theorem C.
For the convenience of the reader in Appendix we repeat the relevant 
material from \cite{CLW-uniform} without proofs. 
\section{Notation}
Throughout the paper $N$ will denote
the spatial dimension.  
$B_r(x_0)=\{x\in \R^N, s.t. \ |x-x_0|<r\}$ denotes the open ball of 
radius $r>0$ centered at $x_0\in \R^N$.
The $s$-dimensional Hausdorff measure is denoted by
$\H^s$, the unit sphere by $\S^{N-1}\subset \R^{N}$, and the
characteristic function of the set $D$ by $\chi_{D}$. We also
\[M=\int_0^1\beta(t)dt.\]
Sometimes we will denote $x=(x_1, x')$ where $x'\in \R^{N-1}$.
For given function $v$, we will denote $v^+=\max(0, v)$ and
$v^-=\max(0, -v)$. Finally, we say that $v\in C^{0,1}_{loc}(\mathcal D)$
if for every $\mathcal D'\Subset \mathcal D$, there is a constant
$L(\mathcal D')$ such that
\[|v(x)-v(y)|\leq L(\mathcal D')|x-y|, \quad \forall x, y\in \mathcal D.\]
If $v\in C^{0, 1}_{loc}(\mathcal D)$ then we say that
$v$ is locally Lipschitz continuous in $\mathcal D$.
For $x=(x_1, \dots, x_N)$ and fixed $x_0\in \R^N$
$(x-x_0)_1^+$ denotes the positive part of the first coordinate of $x-x_0$.
If $u(x_0)=0$ then  $(u(x))_r=u(x_0+rx)/r, r>0$ denotes the scaled function at $x_0$.
For given $r_j\to 0$ the sequence $(u(x))_{r_j}$ 
is called a blow-up sequence and its limit $u_0$ a blow-up of $u$ at $x_0$. 

\section{Monotonicity formula of Spruck}
It is convenient to work with  a weaker definition of non-degeneracy which only assures
that the blow-up does not vanish identically. 
\begin{defn}
We say that $u$ is degenerate at $x_0\in \fb u$ if $\liminf\limits_{r\to 0}\frac1r \fint\limits_{B_{r}(x_0)}u^+=0$.
\end{defn}
Observe that $u^+(x)=o(|x-x_0|)$ near the degenerate point $x_0$ because $u^+$ is subharmonic.

It is known that the solutions of \eqref{pde-0} are locally Lipschitz continuous, see Appendix Proposition \ref{prop1}. Consequently, there is a subsequence $\e_j\to 0$ such that 
$u_{\e_j}\to u$ locally uniformly. Furthermore,  $u$ is a stationary  point of the Alt-Caffarelli problem in some weak sense and the blow-up of $u$ can be approximated by some scaled family of 
solutions to \eqref{pde-0}, see Appendix Propositions \ref{prop-1st-blow} and \ref{prop-2nd-blow}.
\begin{prop}\label{prop1-1}
Let $u$ be a limit of some sequence  $u_{\e_j}$ as in Proposition \ref{prop-compactness}.
Then any blow-up of $u$ at a non-degenerate point is a homogeneous function of degree one.
\end{prop}
\begin{proof}
To fix the ideas we assume that $0\in \fb u$  is a non-degenerate point. 
We begin with writing the Laplacian in polar coordinates
\begin{equation}
\lap u=u_{rr}+\frac{N-1}ru_r+\frac1{r^2}\lap_{\mathbb S^{N-1}}u
\end{equation}
and then introducing the auxiliary function
\begin{equation}
v(t, \sigma)=\frac{u(r, \sigma)}{r}, \quad r=e^{-t}.
\end{equation}
A straightforward computation yields
\begin{eqnarray*}
v_t&=&-u_r+v,\\
v_\sigma&=&\frac{u_\sigma}r,\\
v_{tt}&=&u_{rr}r+v_t,\\
\lap_{\mathbb S^{N-1}}v&=&\frac1r\lap_{\mathbb S^{N-1}}u,
\end{eqnarray*}
where, with some abuse of notation, $v_{\sigma}$ denotes the gradient of $v$ computed on the sphere.
Rewriting the equation $\Delta u_{\e}=\beta_\e(u_\e)$ in $t$ and $\sigma$ derivatives we obtain
\begin{equation*}
\frac1r[(N-1)(v-\p_tv_\e)+\p^2_tv_\e-\p_tv_\e+\lap_{\mathbb S^{N-1}}v_\e]=
\frac1\e\beta\left(\frac{r}\e v_\e\right).
\end{equation*}
\medskip
Next, we multiply both sides of the last equation by $\p_tv_\e$ to get
\begin{equation}\label{v-t}
{\p_tv_\e}[(N-1)(v-\p_tv_\e)+\p^2_tv_\e-\p_tv_\e+\lap_{\mathbb S^{N-1}}v_\e]=
\p_t v_\e\frac r\e\beta\left(\frac{r}\e v_\e\right).
\end{equation}
The right hand side of \eqref{v-t} can be further transformed as follows
\begin{eqnarray*}
\frac r\e\beta\left(\frac{e^{-t}}\e v_\e\right)\p_tv_\e
&=&\beta\left(\frac{e^{-t}}\e v_\e\right)\left[\frac{e^{-t}}\e \p_tv_\e-\frac{e^{-t}}\e v_\e\right]+
\beta\left(\frac{e^{-t}}\e v_\e\right)\frac{e^{-t}}{\e}v_\e\\
&=&\p_t\int_0^{\frac{e^{-t}}{\e}v_\e}\beta(s)ds+\beta\left(\frac{e^{-t}}\e v_\e\right)\frac{e^{-t}}{\e}v_\e\\
&=&\p_t\B \left(\frac{e^{-t}}{\e}v_\e\right)+\beta\left(\frac{e^{-t}}\e v_\e\right)\frac{e^{-t}}{\e}v_\e\\
&\equiv& I_1.
\end{eqnarray*}
It is important to note that by our assumption \eqref{beta} the last term is nonnegative, in other words
\begin{equation}\label{nonneg}
\beta\left(\frac{e^{-t}}\e v_\e\right)\frac{e^{-t}}{\e}v_\e\ge 0.
\end{equation}
Moreover, we have
\begin{eqnarray*}
I_2&\equiv&[(N-1)v_\e-N\p_tv_\e+\p_t^2v_\e+\lap_{\mathbb S^{N-1}}v_\e]\p_tv_\e\\
&=& (N-1)\p_t\left( \frac{v_\e^2}2\right)-Nv_t^2+\p_t\left(\frac{(\p_tv_\e)^2}2\right)+\p_tv_\e\lap_{\mathbb S^{N-1}}v_\e.
\end{eqnarray*}
Next we integrate the identity
$$I_2=rI_1$$
over $\mathbb S^{N-1}$ and then over $[T_0, T]$
in order to get
\begin{eqnarray*}
\left.(N-1)\int_{\mathbb S^{N-1}} \frac{v_\e^2}2\right|_{T_0}^T-N\int_{T_0}^T\int_{\mathbb S^{N-1}}(\p_tv_\e)^2+\left.\int_{\mathbb S^{N-1}}\frac{(\p_tv_\e)^2}2\right|_{T_0}^T+\int_{T_0}^T\int_{\mathbb S^{N-1}}\p_tv_\e\lap_{\mathbb S^{N-1}}v_\e\\
= \int_{\mathbb S^{N-1}}\left. \mathcal B \left(\frac{e^{-t}}{\e}v_\e\right)\right|^T_{T_0}
+\int_{T_0}^{T}\int_{\S^{N-1}}\beta \left(\frac{r}{\e}v_\e\right)\frac r\e v_\e.
\end{eqnarray*}

Note that
\begin{eqnarray}
\int_{T_0}^T\int_{\mathbb S^{N-1}}\p_tv_\e\lap_{\mathbb S^{N-1}}v_\e=-\left.\frac12\int_{\mathbb S^{N-1}}
|\nabla_\sigma v_\e|^2\right|_{T_0}^T.
\end{eqnarray}
Rearranging the terms and utilizing \eqref{nonneg} we get the identity
\begin{eqnarray}\nonumber
N\int_{T_0}^T\int_{\mathbb S^{N-1}}(\p_tv_\e)^2+\int_{T_0}^{T}\int_{\S^{N-1}}\beta \left(\frac{r}{\e}v_\e\right)\frac r\e v_\e&=&\left.(N-1)\int_{\mathbb S^{N-1}} \frac{v_\e^2}2\right|_{T_0}^T+
\left.\int_{\mathbb S^{N-1}}\frac{(\p_tv_\e)^2}2\right|_{T_0}^T\\\label{crazy-iden}
&&-\left.\frac12\int_{\mathbb S^{N-1}}
|\nabla_\sigma v_\e|^2\right|_{T_0}^T
- \int_{\mathbb S^{N-1}}\left. \mathcal B \left(\frac{e^{-t}}{\e}v_\e\right)\right|^T_{T_0}.
\end{eqnarray}

From here it follows that 
\begin{equation}
\int_{T_0}^T\int_{\mathbb S^{N-1}}\left(\p_tv_\e\right)^2\le C
\end{equation}
where $C$ depends on $\|\na u_\e\|_\infty, M, N$ but not on $\e, T_0$ or $T$.

Letting $\e\to 0$ 
we conclude
\begin{equation}
\int_{T_0}^T\int_{\mathbb S^{N-1}}\left(\p_tv\right)^2\le C
\end{equation}
where $v(t,\sigma)=\frac{u(r, \sigma)}r$.
But $\p_t v=-u_r+\frac{u}r$  implying that
\begin{equation}
\int_{T_0}^\infty\int_{\mathbb S^{N-1}}\left(u_r-\frac{u}r\right)^2dtd\sigma\le C.
\end{equation}
The proof of Theorem A follows if we note that $-u_r+\frac{u}r=0$ is the Euler equation for the homogeneous functions
of degree one.
\end{proof}

In the proof of Proposition \ref{prop1-1} we used Spruck's original computation \cite{Spruck}.
The identity \eqref{crazy-iden} can be interpreted as a local energy 
balance for $u_\e$. Moreover, using \eqref{crazy-iden} we can construct a monotone quantity which has some remarkable properties.
\begin{cor}\label{cor-mon}
Suppose $0\in \fb u$ and let $(r, \sigma), \sigma\in \S^{N-1}$ be the spherical coordinates. Introduce 
\begin{equation}\label{S-eps}
S_\e(r)=\int_{\S^{N-1}}\left\{2\B(u_\e(r, \sigma)/\e)+\frac1{r^2}|\na_\sigma u_\e|^2-(N-1)\frac{u_\e^2(r, \sigma)}{r^2}-
\left(\p_ru_\e(r,\sigma)-\frac{u_\e(r, \sigma)}{r}\right)^2\right\}d\sigma.
\end{equation}
\begin{itemize}
\item Then $S_\e(r)$ is nondecreasing in $r$.
\item  Moreover, if $u_{\e_j}\to u$ for some subsequence $\e_j\to 0$, then
$S_{\e_j}(r)\to S(r)$ for a.e. $r$ where
\begin{equation}\label{S}
S(r)=\int_{\S^{N-1}}\left\{2M\I{u}+\frac1{r^2}|\na_\sigma u|^2-(N-1)\frac{u^2(r, \sigma)}{r^2}-
\left(\p_ru(r,\sigma)-\frac{u(r, \sigma)}{r}\right)^2\right\}d\sigma.
\end{equation}
In particular, $S(r)$ is nondecreasing function of $r$.
\item $S(r)$ is constant if and only if $u$ is a homogenous function of degree one.
\end{itemize}
\end{cor}

\begin{proof}
By setting $r_1=e^{-T}, r_2=e^{-T_0}$ and noting that
$r_1<r_2$ if $T>T_0$ we obtain from \eqref{crazy-iden}
\[S_\e(r_2) -S_\e(r_1)=2N\int_{T_0}^T\int_{\mathbb S^{N-1}}(\p_tv_\e)^2+2\int_{T_0}^{T}\int_{\S^{N-1}}\beta \left(\frac{r}{\e}v_\e\right)\frac r\e v_\e\ge 0\]
where we applied \eqref{nonneg} and hence the first claim follows. The second part follows from Propositions \ref{prop1} and \ref{prop-compactness}.
Indeed, integrating $S_\e(r)\leq S_\e(r+t), t\ge 0$ over $[r_1-\delta, r_1+\delta]$ we infer
\[\frac1{2\delta}\int_{r_1-\delta}^{r_1+\delta}S_\e(r)dr\le \frac1{2\delta}\int_{r_1-\delta}^{r_1+\delta}S_\e(r+t)dr.\]
Then first letting $\e\to 0$ and utilizing Proposition \ref{prop-compactness} together with \eqref{grad-est} and then sending $\delta\to 0$ we infer that
$S(r)$ is nondecreasing for a.e. $r$.
Finally the last part follows as in the proof of Proposition \ref{prop1-1}.
\end{proof}

As one can see we did not use the Pohozhaev identity as opposed to the monotonicity formula in \cite{Weiss}.
Spruck's monotonicity formula enjoys a remarkable property.

\begin{lem}\label{lem-semi}
Let $u$ be as in Proposition \ref{prop1-1}. Set $S(x_0, r, u)$ for $S(r)$ defined by  the sphere centered at $x_0\in\fb u$.
Suppose $x_k\in \fb u$ such that $x_k\to x_0$ then
\[\limsup_{x_k\to x_0}S(x_k, 0, u)\le S(x_0, 0, u).\]
\end{lem}

\begin{proof}
For given $\delta>0$ there is $\rho_0>0$ such that $S(x_0, \rho, u)\leq S(x_0, 0, u)+\delta$ whenever
$\rho<\rho_0$. Fix such    $\rho$ and choose $k$ so large that $S(x_k,\rho, u)<\delta+S(x_0, \rho, u)$.
From the monotonicity of $S(x_k, \rho, u)$ it follows that 
\begin{eqnarray*}
S(x_k,0,u )&\leq& S(x_k, \rho, u)\le \delta+S(x_0, \rho, u)\\
&\le& 2\delta +S(x_0, 0, u).
\end{eqnarray*}
First letting $x_k \to x_0$ and then $\delta \to 0$ the result follows.
\end{proof}

\begin{lem}\label{lem-s-small}
Let $S$ be the monotone quantity in \eqref{S}. Then the following holds:
\begin{itemize}
\item[(i)]
$s(x_0, R, u)=\frac1{R^N}\int_0^R r^{N-1}S(x_0,r, u)dr$ is monotone non-decreasing and 
$$\frac{d}{dR}s(x_0, u, R)=\frac1{R^{N+1}}\int_0^Rr^NS'(x_0,u, r)dr\ge 0.$$
\item[(ii)] If the solution $u\ge 0$ is degenerate at $x_0\in \fb u$ then 
the set $\{u>0\}$ has well defined Lebesgue density 
$\Theta(x_0, \po u)$ equal to 
$\frac1{2M|B_1|}s(x_0, 0, u)=\frac1{2M|B_1|}\lim_{R\to 0} s(x_0, R, u)$.
\item[(iii)] 
Suppose $x_k\in \fb u$ such that $x_k\to x_0$ then
\[\limsup_{x_k\to x_0}s(x_k, 0, u)\le s(x_0, 0, u).\]
\end{itemize}
\end{lem}
\begin{proof}
It is easy to compute 
\begin{eqnarray*}
s'(x_0, R, u)
&=&
-\frac N{R^{N+1}}\int_0^R r^{N-1}S(x_0,r, u)dr+\frac{S(x_0, R, u)}{R}\\
&=&
-\frac{S(x_0, R, u)}{R}+\frac1{R^{N+1}}\int_0^Rr^NS'(x_0, u, r)dr+\frac{S(x_0, R, u)}{R}\\
&=&
\frac1{R^{N+1}}\int_0^Rr^NS'(x_0, r, u)dr.
\end{eqnarray*}

To prove the second claim notice that at degenerate point $x_0$ we have $u(x)=o(|x-x_0|)$ by virtue of the subharmonicity of $u$. Consequently $\fint_{B_R(x_0)}|\na u|^2=o(1)$ as $r\to 0$  by virtue of the Caccioppoli inequality.
Therefore the only surviving term in $S$ comes from $2M\I u$. 
The proof of the last  claim is analogous to that of Lemma \ref{lem-semi}.\end{proof}

\begin{lem}\label{lem-vigenik}
Let $0\in \fb u$ and assume that $u_0=rg(\sigma), \sigma\in \S^{N-1}$ is a blow-up limit of $u$ 
at $0$ which is homogeneous function of degree one. Then 
\[
\int_{\S^{N-1}}|\na_\sigma g|^2-(N-1)\int_{\S^{N-1}}g^2\quad 
\begin{array}{rrr}
=0 &\quad  \text{if $\fb u$ is flat at $0$},\\
\le 0&\quad \text{otherwise}.
\end{array}
\]
\end{lem}

\begin{proof}
Let $(r, \sigma)$ be the spherical coordinates then the Laplacian takes the form 
$\lap u_\e=\p^2_{rr}u_\e+\frac{N-1}r\p_r u_\e+\frac1{r^2}\lap_{\S^{N-1}} u_\e$. Multiply both sides of $\Delta u_\e$ by $r^{N-1}u_\e$ and integrate  over 
$[0, R]\times \S^{N-1}$ to get 

\begin{eqnarray*}
I_1(u_{\e_j})
&:=&
\int_0^R\int_{\S^{N-1}} u_\e \p^2_{rr}u_\e  r^{N-1}d\sigma dr\\
&=&
R^{N-1} \int_{\S^{N-1}} u_\e \p_r u_\e -\int_0^R\int_{\S^{N-1}}\left[(\p^2_{r}u_\e)^2r^{N-1}+(N-1)\p_r u_\e u_\e r^{N-1}\right]d\sigma dr,
\\
I_2(u_{\e_j})
&:=&
\int_0^R\int_{\S^{N-1}} u_\e \p^2_{r}u_\e  r^{N-2}d\sigma dr=R^{N-2} \int_{\S^{N-1}} \frac{(u_\e)^2}2 -(N-2)\int_0^R\int_{\S^{N-1}}\left[\frac{(u_\e)^2}2r^{N-3}\right]d\sigma dr,
\\
I_3(u_{\e_j})
&:=&
\int_0^R\int_{\S^{N-1}} \lap_{\S^{N-1}} u_\e u_\e  r^{N-3}d\sigma dr=-
\int_0^R \int_{\S^{N-1}} 
|\na_\sigma u_\e|^2 
r^{N-3}d\sigma dr.
\end{eqnarray*}
Choosing a converging sequence $u_{\e_j}$ and 
letting $\e_j\to 0$ we get by virtue of Proposition \ref{prop-compactness}
\begin{equation*}
\lim_{\e_j\to 0}\int_{B_R}\beta_{\e_j}u_{\e_j}=\lim_{\e_j\to 0}
\left[I_1(u_{\e_j})+(N-1)I_2(u_{\e_j})+I_3(u_{\e_j})\right]\to 
I_1(u)+(N-1)I_2(u)+I_3(u).
\end{equation*}
Suppose that $u_{R_k}$ is a blow-up sequence at the origin and $u_{R_k}\to u_0=rg(\sigma)$ then 
\[
I_1(u_0)=R^N\int_{\S^{N-1}} g^2-\frac{R^N}{N}\int_{\S^{N-1}}g^2-\frac{N-1}{N}R^N\int_{\S^{N-1}}g^2=0
\]
and 
\[
I_2(u_0)=R^N\int_{\S^{N-1}}\frac{g^2}2-\frac{N-2}{N}R^N\int_{\S^{N-1}}\frac{g^2}2=\frac{R^N}N\int_{\S^{N-1}}{g^2}.
\]
By Proposition 
\ref{prop-1st-blow} and \eqref{eq-lap-scale} 
there is a sequence $\delta_j\to 0$ such that 
$u_{\delta_j}\to u_0$ and $\lim_{\delta_j\to 0}\int_{B_1}\beta_{\delta_j}u_{\delta_j}\le \|\beta\|_\infty |x\in B_1 : 0<u_{\delta_j}<\delta_j\}|\to 0$
provided that $u$ is flat at $0$. Hence we have 
\[
\lim_{\delta_j\to 0}\int_{B_1}\beta_{\delta_j}u_{\delta_j}=\frac{R^N}N\left[(N-1)\int_{\S^{N-1}}{g^2}-\int_{\S^{N-1}}{|\na_\sigma g|^2}\right]
\quad \begin{array}{rrr}
=0 &\quad  \text{if $\fb u$ is flat at $0$},\\
\ge 0&\quad \text{otherwise}.
\end{array}
\]
\end{proof}

\section{Proof of Theorem A}

The first part of the theorem follows from Proposition \ref{prop1-1}. 
Since $u$ is not degenerate  at the origin then  by Propositions \ref{prop-compactness} and \ref{prop-1st-blow} $u_{\rho_k}(x)\to u_0(x)$ locally
uniformly and by Proposition \ref{prop1-1} $u_0$ is homogeneous of degree one.
Write~$\lap$ in polar coordinates $(r, \theta)$, to obtain that
\begin{eqnarray}\nonumber
 \lap w=\frac 1r \frac\partial{\partial r}
\left(rw_r\right)+
\frac 1{r^2}\frac\partial{\partial \theta }\left(w_\theta\right).
\end{eqnarray}
In particular, writing~$u_0=rg(\theta)$, this yields
a second order ODE for $g$
\begin{equation}\label{plap-polar}
g+\ddot g=0.
\end{equation}
Suppose $g(0)=g(\theta_0)=0, \theta_0\in [0, 2\pi)$, then \eqref{plap-polar} implies that~$g(\theta)=A\sin\theta$ for some~
constant $A$, and  consequently
forcing $\theta_0=\pi$.
Hence, since~$N=2$, we obtain that~$u_0$ must be linear,
in other words the free boundary $\fb{u_0}$ is everywhere flat.
This in turn implies that in two dimensions
the singular set of the free boundary $\fb{u_0}$ is empty.
Consequently, $u_0$ is linear in $\{u_0>0\}$ and $\{u_0<0\}$.
From here the parts (2) and (3) of Theorem A follow from Propositions 5.3 and 5.1 of \cite{CLW-uniform}.

So it remains
to check (1). For the elliptic problem the only difference is that
the limit function $M(x)=\lim_{\delta_j\to 0}\mathcal B_{\delta_j}(u_{\delta_j})$
cannot have nontrivial concentration on the free boundary coming from
$\{x_1<0\}$ as opposed to the parabolic case  studied in \cite{CLW-uniform}.
Observe that $\na \mathcal B(u_{\delta_j}/\delta_j)=\na u_{\delta_j}\beta_{\delta_j}(u_{\delta_j})=0$ in $B_1\setminus \{0<u_{\delta_j}<\delta_j\}$. By Proposition \ref{prop-1st-blow} and \eqref{eq-lap-scale} there is sequence $0<\lambda_j\to 0$ such that
$(u_{\e_j})_{\lambda_j}\to u_0, \e_j/\lambda_j\to 0$ and $M(x)=M\I{x_1}+M_0\chi_{\{x_1<0\}}$. It follows from \eqref{domain-var} that
\begin{equation}\label{domain-var-1}
\int_{\{x_1>0\}}M\p_1\phi+\int_{\{x_1<0\}}M_0\p_1\phi=\int_{\{x_1>0\}}\frac{\alpha^2}2\p_1\phi, \quad \forall \phi\in C_0^\infty(B_1).
\end{equation}
After integration by parts we obtain $M_0\int_{-1}^1\phi(0, x_2)dx_2=\left(M-\frac{\alpha^2}2 \right)\int_{-1}^1\phi(0, x_2)dx_2$. This yields
\[
M_0=M-\frac{\alpha^2}2.
\]

Next we claim that $M_0=0$. Suppose $M_0>0$ then 
the set $I_0:=\{t\in \R : \mathcal B(t)=M_0\}\not =\emptyset$ and there is 
$a\in(0,1)$ such that $I_0\subset [a, 1]. $
Since $\mathcal B(t)$ is continuous and non-decreasing then it follows that there is $0<a_0<a$ such that 
$u_{\delta_j}(x)/\delta_j\in[a_0, 1] $ provided that 
$j$ is sufficiently large.
Let $\mathcal C=\{x : u_{\delta_j}(x)/\delta_j\in[a_0, 1]|\}\cap  \{x_1<0\}\cap B_1$. 
Then,  $\mathcal C\subset \{x\in B_1 : a_0\delta_j\le u_{\delta_j}(x)\le \delta_j\}\subset \{0<u_{\delta_j}<2\delta_j\}\cap B_1$.
But $|\{0<u_{\delta_j}<2\delta_j\}\cap B_1|\to 0$ which implies that 
$M_0$ cannot be positive. 
\qed

\section{The structure  of the free boundary of blow-ups in $\R^3$}
In this section we assume that $u\ge 0$ is a limit of $u_{\e_j}$ solving
\eqref{pde-0} for some sequence $\e_j\to 0$, $u$ is non-degenerate at some $y_0\in \fb{u}$ and $u_0$ is a blow-up of $u$
at $y_0$. Note that by Corollary \ref{cor-mon} $u_0$ is homogenous function of degree one.
If $u_0$ is not a minimizer then it is natural to expect that the solutions of \eqref{pde-0} develop
singularities in $\R^N, N\ge 3.$

We first prove a non-degeneracy result. 
\begin{lem}\label{lem-beta-non-deg}
Let $x_0\in \fb {u_0}$ be a free boundary point such that there is a ball $B\subset \{u_0=0\}$ touching 
$\fb {u_0}$ at $x_0$ and $\Theta(x_0, \po{u_0})\ge \frac12$. Then $u_0$ is non-degenerate at $x_0$ and 
\[u_0(x)=\sqrt{2M}(x-x_0)^++o(x-x_0).\]
\end{lem}

\begin{proof}
Let $(u_0)_r=\frac{u_0(x_0+rx)}r$. There is $r_0$ such that 
\begin{equation}\label{cube-1}
(u_0)_r=0\quad \text{ in} \ \ \{x_1<-\delta\}\cap Q_1,\quad \forall r\le r_0
\end{equation}
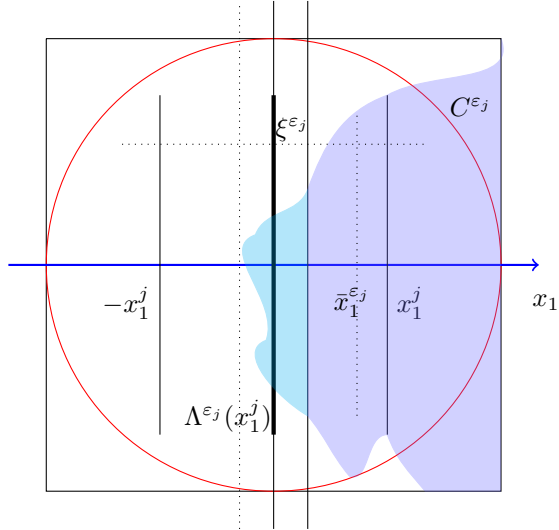
\begin{figure}
\begin{center}
 \begin{tikzpicture}

    \coordinate (tl) at (-3,3) ; 
    \coordinate (tr) at (3,3) ; 
        \coordinate (bl) at (-3,-3) ; %
    \coordinate (br) at (3,-3) ; %
            \coordinate (A) at (1.5,2.25) ; 
        \coordinate (B) at (1.5,-2.25) ; 
\coordinate (C) at (-1.5,2.25);
\coordinate (D) at (-1.5,-2.25);

\draw (tl)--(tr)--(br)--(bl)--(tl);

\draw[red] (0,0) circle (3cm);

 \draw[->, blue, thick] (-3.5, 0)-- (3.5, 0) ;
 
\draw (0, -3.5)--(0, 3.5);
\draw[dotted] (-0.45, -3.5)--(-0.45, 3.5);
\draw (0.45, -3.5)--(0.45, 3.5);

\draw (A)-- (B);
\draw (C)--(D);
\draw[ultra thick](0, 2.25)--(0, -2.25);
 \node[right] at (1.5,-0.5) {$x_1^j$};
  \node[left] at (-1.5,-0.5) {$-x_1^j$};
   \node[left] at (0.1, -2) {$\Lambda^{\e_j}(x_1^j)$};

\filldraw[blue!60!,opacity=.3] (tr) to[out=-70,in=30] (A) 
to[out=200,in=70] (0.45,1)--(0.45,-2)
to[out=-50,in=130] (1, -2.8)to[out=-20,in=180](B) to[out=-70,in=130] (2, -3)
--(3, -3);

\filldraw[cyan,opacity=.3] 
(0.45,1) to[out=250,in=70] (-0.24,0.4)
to[out=200,in=70] (-0.1, -1.0)
to[out=200,in=150] (0.45, -2);

\draw[dotted] (-2, 1.6)-- (2, 1.6); 
\draw[dotted] (1.1, -2)-- (1.1, 2); 

\node[right] at (2.2, 2.1) {$C^{\e_j}$}; 

\node[right] at (-0.1, 1.8) {$\xi^{\e_j}$};
\node[left] at (1.4, -0.5) {$\bar x_1^{\e_j}$}; 
\node[left] at (3.9, -0.5) {$x_1$}; 

\end{tikzpicture}
\end{center}
\caption{The construction of the point $(\bar x_1^{\e_j}, \xi^{\e_j})$. 
The purple region is $C^{\e_j}$. }
\label{fig-1}
\end{figure}
for some small $\delta>0$, where $Q_1=(-1, 1)^3$
is the unit cube. 
Moreover, there is $\widehat r_0>0$ such that 
\begin{equation}\label{cube-2}
\frac{|\{(u_0)_r>0\}\cap \{x_1>0\}\cap B_1|}{|B_1|}>\frac12-\delta, \quad \forall r\le \widehat r_0.
\end{equation}
Fix $r$ with these two  properties \eqref{cube-1} and \eqref{cube-2}. 
There exists $\gamma>0$ such that 
\begin{equation}
\frac{|\{(u_0)_r>\gamma\}\cap \{x_1>0\}\cap B_1|}{|B_1|}>\frac12-2\delta
\end{equation}
Denote $v^{\e_j}=(u_{\e_j})_r$ where $u_{\e_j}\to u_0$ (see Proposition \ref{prop-1st-blow}) and 
$A^{\e_j}=\{v^{\e_j}>\frac\gamma2\}\cap \{x_1>0\}\cap B_1$.
Since $v^{\e_j}\to (u_0)_r$ uniformly (see Proposition \ref{prop-compactness}) it follows that there is $j_0(r)$ such that  for 
$j\ge j_0(r)$ we have 
\begin{equation}
|A^{\e_j}|>|B_1|(\frac{1}2-2\delta).
\end{equation}
 Let 
$B^{\varepsilon_j}=\left\{ x_{1}\in \left( -1,-\delta \right) \right\} \cap Q_{1}$
and 
$-B^{\varepsilon_j}=\left\{ x_{1}\in \left( \delta ,1\right) \right \} \cap Q_{1}$.
Let $C^{\varepsilon_j}=A^{\varepsilon_j}\cap \left( -B^{\varepsilon_j}\right)$
then we have that  
\[
\left| C^{\varepsilon_j}\right| \geq |B_1|(\dfrac {1}{2}-2\delta)  > 0.
\]

Denote $\Lambda ^{\varepsilon _{j}}\left( x_{1}\right) =\{ x' : \left( x_{1},x' \right) \in C^{\varepsilon_j}\}$
and 
$f^{\varepsilon_j}\left( x_{1}\right) =\left| \Lambda ^{\varepsilon_j}\left( x_{1}\right) \right|.$
We claim that $|\{x_1 : f^{\e_j}(x_1)>|B_1|(\frac12-3\delta)\}|>0$. Indeed, if the claim fails then we have 
\[
|B_1|(\frac12-2\delta)\le |C^{\e_j}|=\int_\delta^1f^{\e_j}(x_1)dx_1\le |B_1|(\frac12-3\delta)
\]
which is a contradiction. 

Hence there is $x^{\e_j}_1\in(\delta, 1)$ such that $f^{\e_j}(x_1^{\e_j})>|B_1|(\frac12-3\delta)$.
Now choose $0<a'<a<b<b'<1$ such that 
\[\beta(s)>\kappa, \forall s\in [a', b'].\]

Let $\e_j'=\frac{\e_j}r$. We claim that there is $\xi^{\e_j}\in \Lambda^{\e_j}$ and $\bar x_1^{\e_j}$ such that 
\[
\frac{v^{\e_j}}{\e'_j}(x_1^{\e_j}, \xi^{\e_j})\in (a, b).
\]
Indeed, for sufficiently large $j$ we have 
\[
a > \dfrac {v^{\e_j}} {\varepsilon '_j} ( -x^{\e_j}_{1},x' ) =0, \quad \forall x'\in 
\Lambda ^{\varepsilon_ j}( x^{\e_j}_{1}) 
\]
and 

\[
\dfrac {v^{\e_j}} {\varepsilon '_j} ( x^{\e_j}_{1},x' ) >\frac{\gamma}{2\e_j'}>b, \quad \forall x'\in 
\Lambda ^{\varepsilon_ j}( x^{\e_j}_{1}) 
\]
provided that $j>j_1(r)$, see Figure \ref{fig-1}.
Hence form the mean value theorem we see that the claim is true. 
From the uniform Lipschitz continuity of the functions $v^{\e_j}$ it follows that 
there is a constant $c_0>0$ such that 
\[
\dfrac {v^{\e_j}} {\varepsilon '_j} ( x_{1},x' ) \in(a', b'), \quad \text{if}\  |x_1-\bar x_1^{\e_j}|<\e_j'c_0, x'\in 
\Lambda ^{\varepsilon_ j}( x^{\e_j}_{1}). 
\] 
Consequently we have that 
\[
\int_{B_1}\beta_{\e_j'}(v^{\e_j})\ge \frac{\kappa}{\e_j'}\int_{|x_1-\bar x_1^{\e_j}|<\e_j'c_0}|\Lambda^{\e_j}(x_1)|dx_1\ge \frac{\kappa}{\e_j'}(1-3\delta)2c_0\e_j'=2\kappa|B_1|(1-3\delta)c_0:=\widetilde C.
\]
Now the non-degeneracy follows from the proof of Part II of Theorem 6.3 \cite{CLW-uniform}.
The asymptotic expansion follows from Theorem A and Proposition \ref{prop1-1}.
\end{proof}

\begin{figure}\label{fig-2}
\begin{center}
\begin{tikzpicture}
\filldraw[fill=cyan, opacity=.2, ] 
    (-2.5,2.5)-- (0,0) -- (-2.5,-2.5) -- (-2.5, 2.5); 
    \filldraw[fill=cyan, opacity=.2] 
    (2.5,2.5)-- (0,0) -- (2.5,-2.5) -- (2.5, 2.5); 
    \node[right] at (1,0){$u>0$};
        \node[left] at (-1, 0){$u>0$};
\draw[ultra thick](-2.5, 2.5)--(2.5,-2.5);
\draw[ultra thick](-2.5, -2.5)--(2.5,2.5);
\node[below] at (0,-0.2) {$O$};
\end{tikzpicture}
\end{center}
\caption{Possible self-crossing free boundary which fails to satisfy the conditions of Lemma \ref{lem-beta-non-deg}.}\label{fig-2}
\end{figure}
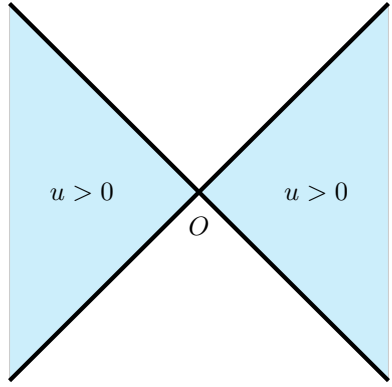
\begin{remark}
Note that under weaker assumption $\Theta(x_0, \{u_0>0\})>0$ the argument in the proof of Lemma \ref{lem-beta-non-deg} still works.
However for the self-crossing free boundary \cite{Weiss} (see Figure \ref{fig-2})
the assumptions of  Lemma  \ref{lem-beta-non-deg}. 
\end{remark}

As an immediate corollary we have 
\begin{cor}
Let $x_0\in \fb {u_0}$ be a point of reduced boundary. Then $u_0$ is non-degenerate at $x_0.$ 
\end{cor}
\begin{proof}
Suppose that $0\in \fb{u_0}$ and $\fb{(u_0)_r}\subset B_2\cap \{|x\cdot e|<\e\}$
or some unit vector $e$ and small $\e>0$. Here $(u_0)_r=\frac{u_0(rx)}r$. 
Consider the family of balls $B_{1/2}(et), t\in[-\e, \e]$. Then there is 
$t_\e\in [-\e, \e]$ such that $B_{1/2}(et_\e)$ touches the free boundary at some 
point $z_0\in B_1$ provided that $\e$ is sufficiently small.
Let $\nu_0=t_\e e$. 
Introduce the following barrier function 
\[
w(x)=\dfrac {\varphi \left( 1/2\right) -\varphi \left( \left| x-\nu_0\right| \right) }
{\varphi \left( 1/2\right) -\varphi \left( 1\right) }
\sup _{B_{1}(\nu_0)}u_{0}
\]
where $\varphi(|x|)=\frac1{|x|^{N-2}}$. We have 
$\lap(u_0-w)=\lap u_0\ge 0$ in $D=B_1(\nu_0)\setminus B_{\frac12}(\nu_0)$
and $u_0-w\le 0$ on $\partial D$. From the maximum principle
we infer that 
$u_0\le w$ in $D$. But we have that the maximum of $u_0-w$ is realized at $z_0$. Hence 
from the Hopf lemma we get 
$$-\sqrt{2M}=\partial_{\nu_0} u_0(z_0)>\partial_{\nu_0} w(z_0)=-|\varphi'(1/2)|
\frac{\sup _{B_{1}(\nu_0)}u_{0}}{\varphi \left( 1/2\right) -\varphi \left( 1\right) }
$$
or 
\[
\sup _{B_{1}(\nu_0)}(u_{0})_r\ge \sqrt{2M}
\frac{\varphi(1/2)-\varphi(1)}{|\varphi'(1/2)|}.
\]
\end{proof}

In the following definition we denote 
$\Omega^+(u)=\{u>0\}$ and~$\Omega^-(u)=\{u<0\}$.
Moreover, let 
\begin{equation}\label{FB-cond-G}
G(u^+_\nu,u^-_\nu):=(u_\nu^+)^2-(u_\nu^-)^2-2M
\end{equation} 
where~$u^+_\nu$ and~$u^-_\nu$ are the normal derivatives in the inward direction $\nu$
to~$\partial \Omega^+(u)$ and~$\partial \Omega^-(u)$, respectively. 
For more details see Definition 2.4 \cite{CS-book}.
\begin{defn}\label{def:visc}
Let~$\Omega$ be a bounded domain of~$\R^N$ and let~$u$ 
be a continuous function in~$\Omega$. We say that~$u$ is a viscosity solution
in~$\Omega$ if
\begin{itemize}
\item[i)] $\Delta u=0$ in~$\Omega^+(u)$ and~$\Omega^-(u)$,
\item[ii)] along the free boundary~$\fb u$, $u$ satisfies the free boundary condition, in the sense that:
\begin{itemize}
\item[a)] if at~$x_0\in\fb u$ there exists a ball~$B\subset\Omega^+(u)$
such that~$x_0\in \partial B$ and
\begin{equation}\label{visc1}
u^+(x)\ge\alpha\langle x-x_0,\nu\rangle^+ + o(|x-x_0|), \ {\mbox{ for }} x\in B,
\end{equation}
\begin{equation}\label{visc2}
u^-(x)\le\beta\langle x-x_0,\nu\rangle^- + o(|x-x_0|), \ {\mbox{ for }} x\in B^c,
\end{equation}
for some $\alpha>0$ and~$\beta\ge0$, with equality along every non-tangential domain,
then the free boundary condition is satisfied
$$ G(\alpha,\beta)\le 0, $$
\item[b)] if at~$x_0\in\fb u$ there exists a ball~$B\subset\Omega^-(u)$
such that~$x_0\in \partial B$ and
$$ u^-(x)\ge\beta\langle x-x_0,\nu\rangle^- + o(|x-x_0|), \ {\mbox{ for }} x\in B, $$
$$ u^+(x)\le\alpha\langle x-x_0,\nu\rangle^+ + o(|x-x_0|), \ {\mbox{ for }} x\in\partial B, $$
for some $\alpha\ge0$ and~$\beta>0$, with equality along every non-tangential domain,
then
$$ G(\alpha,\beta)\ge0. $$
\end{itemize}
\end{itemize}
\end{defn}

In our case $\beta=0$ and we have only $u^+$. However
one has to check that the free boundary conditions  $a)$ and $b)$ in Definition \ref{def:visc}
are satisfied.

\begin{lem}\label{lem-u-0-visc}
Let $u_0$ be a blow-up of $u$ at some non-degenerate point such that 
$\Theta (x, \po{u_0})\ge \frac12$ for every $x\in \fb{u_0}$. Then $u_0$
is a viscosity solution in the sense of Definition \ref{def:visc}.
\end{lem}

\begin{proof}
We have to show that the properties $a), b)$ in Definition \ref{def:visc}
 hold. Suppose that $B\subset \{u_0>0\}$ touches $\fb{u_0}$ at some point $x_0$. Then it follows from Hopf's lemma 
that $u_0$ is non-degenerate at $x_0$. Consequently, if $u_{00}$
is a blow-up at $x_0$ then by Theorem A $u_{00}(x)=\alpha x_1^+$
after some rotation of coordinate system. Moreover $0<\alpha\le \sqrt{2M}$. Hence $G(\alpha, 0)\le 0$. 

Now suppose that $B\subset \{u_0=0\}$ and $B$ touches  
$\fb {u_0}$ at $z_0$. By Lemma \ref{lem-beta-non-deg} $u_0$
is non-degenerate at $z_0$. Theorem A implies that any blow-up $u_{00}$ of 
$u_0$ at $z_0$ must  be $u_{00}(x)= \sqrt{2M}x_1^+$ after some rotation of coordinates. Hence $G(\sqrt{2M}, 0)\ge 0$.  
\end{proof}

\subsection{Properties of $\fb{u_0}$}
We want to study the properties of $g$. We first prove a Bernstein-type result which is a simple consequence of a
refinement of Alt-Caffarelli-Friedman monotonicity formula \cite{ACF}, \cite{CKS}.

\begin{lem}\label{lem-cap}
Let $u\ge 0$ be a limit of solutions to  \eqref{pde-0}. 
Let $u_0=rg(\sigma), \sigma\in \S^{N-1}$ be a
nontrivial blow-up of $u$
at some free boundary point. If there is a hemisphere containing $\supp g$ then the graph of 
$u_0$ is a half-plane.
\end{lem}
\begin{proof}
Without loss of generality we assume that $\supp g\subset \S^{N-1}_+=\{X\in \S^{N-1}, x_N\ge 0\}$.
Let $v(x_1, \dots, x_N)=u(x_1, \dots, -x_N)$ be the reflection of $u$ with respect to the
hyperplane $x_N=0$. Then $v$ is nonnegative subharmonic function
satisfying the requirements of Lemma 2.3 \cite{CKS}. Thus
\[\Phi(r)=\frac1{r^4}\int_{B_r}\frac{|\na u_0|^2}{|x|^{N-2}}\int_{B_r}\frac{|\na v|^2}{|x|^{N-2}}\]
is nondecreasing in $r$. Moreover
\[\Phi'(r)\ge\frac{2\Phi(r)}r A_r, \quad A_r=\frac{C_N}{r^{N-1}}Area(\partial B_r\setminus (\supp u_0\cup\supp v)).\]
Thus,  if $\supp g$ digresses from the hemisphere by size $\delta>0$ then 
$A_r\ge c(\delta)>0$. Hence integrating the
differential inequality for $\Phi$ we see that
$\Phi$ grows exponentially which is a contradiction,  since in view of
Proposition \ref{prop-tech} $u_0$ is Lipschitz and hence $\Phi$ must be bounded.
\end{proof}

It is convenient to define the following subsets of the free boundary
\begin{eqnarray}
\Gamma_{\frac12}&=&\left\{x\in \fb{u_0}\  s.t.\  \Theta(x, \po{u_0})=\frac12\right\},\\
\Gamma_1&=&\{x\in \fb{u_0}\  s.t.\  \Theta(x, \po{u_0})=1\}.
\end{eqnarray}
$\Theta(x, D)$ Denotes the Lebesgue density of $D$ at $x$.
We will see that $\Theta(x, \po{u_0})$ exists at every non-degenerate point and
equals either 1 or $\frac12$.

\begin{lem}\label{lem-asym}
Assume $N=3$. Let $x_0\in \fb{u_0}\setminus\{0\}$ be a non-degenerate free boundary point such that the lower Lebesgue
density $\Theta_*(x_0, \{u_0\equiv0\})>0$.  Then there is a unit vector $\nu_0$ such that
\begin{equation}\label{asym-exp}
u_0(x)=\sqrt{2M}[(x-x_0)\cdot \nu_0]^++o(x-x_0).
\end{equation}
In particular, $x_0\in \Gamma_{\frac12}$
\end{lem}
\begin{proof}
Set $v_k=\frac{u_0(x_0+\rho_k x)}{\rho_k}$, since $u_0$ is non-degnerate at $x_0$
it follows from a customary compactness argument that $v_k\to v$
and by virtue of  Corollary \ref{cor-mon} $v$ is homogeneous function of degree one.
We have
\begin{eqnarray}
\frac{u_0(x_0+\rho_kx)}{\rho_k}&=&u_0(\rho^{-1}_kx_0+x)=\na u_0(\rho^{-1}_kx_0+x)(\rho^{-1}_kx_0+x)=\\
&=& \frac1{\rho_k}\na u_0(x_0+\rho_kx)(x_0+\rho_k x)
\end{eqnarray}
where the last line follows from the zero degree homogeneity of the gradient,
hence \[\rho_kv_k(x)=\na u_0(x_0+\rho_kx)(x_0+\rho_k x)=\na v_k(x)(x_0+\rho_k x).\]
By Lipschitz continuity of $u_0$ it follows that $v_k$
is locally bounded.
Consequently, for a suitable subsequence of $\rho_k$ we have that
$v_{k_j}\to v$ and $\na v(x)x_0=0$.
Without loss of generality we may assume that $x_0$ is on the $x_3$ axis, implying that
$v$ depends only on $x_1$ and $x_2$. Applying Proposition \ref{prop-2nd-blow}
and Corollary \ref{cor-mon} we conclude that $S(x_0, r, u_0)$ is non-decreasing
and thus $v$ must be homogeneous of degree one.
Indeed, there is a sequence
$\delta_j\to 0$ such that $(u_{\e_j})_{\lambda_j}\to v, \delta_j=\e_j/\lambda_j$ by Proposition \ref{prop-2nd-blow}.

Finally, applying Theorem A and the
assumption $\Theta_*(x_0, \{u_0=0\})>0$
we see that $v$ must be a half-plane solution.  It remains to note that
the approximate tangent of $\fb{u_0}$ at $x_0$ is unique and this completes the proof.
\end{proof}


\smallskip
\begin{lem}\label{lem-C1alpha}
We want to show that $\Theta(x, \po{u_0})\ge \frac12$ in some neighbourhood of  $x_0$.
Let $x_0\in \Gamma_{\frac12}$. Then there exists $r_0>0$ such 
that  $B_{r_0}(x_0)\cap\fb {u_0}$ is a $C^{1, \alpha}$ surface.
\end{lem}
\begin{proof}
Let $y_0\in \fb {u_0}$ be a degenerate point. Suppose there is 
$\rho>0$ such that $u_0$ is degenerate at every point of $B_\rho(y_0)\cap \fb{u_0}$. Since $\supp\Delta u_0\subset \fb{u_0}$
then it follows that $u_0\equiv 0$ in $B_\rho(y_0)$.
Consequently, there is a sequence of non-degenerate points $y_k\to y_0$. Note that if $y_k$ is a non-degenerate point 
then by Theorem A the Lebesgue density $\Theta(y_k, \{ u_0>0\})\ge \frac12$.

Let $u_{00}^k$ be a blow-up of $u_0$ at $y_k$.
By Proposition \ref{prop-2nd-blow} for fixed $k$ there are
$\delta_j^k\to 0$ such that $(u_{\e_j^k})_{\lambda_j^k}\to u_{00}^k, \delta_j^k=\e_j^k/\lambda_j^k$.
Thus applying Theorem A it follows that $u_{00}^k$ is a half plane solution or a wedge.

From scaling properties of Spruk's  monotonicity formula and Lemma \ref{lem-s-small} we get
\begin{equation}\label{balag-1}
s(0, y_k, u_0)=s(1, 0, u_{00}^k)=2M\vol(B_1\cap \po{u_{00}^k})=
\left\{
\begin{array}{lll}
2\pi M & \text{if $y_k$ is a wedge point,}\\
\pi M & \text{otherwise}.
\end{array}
\right.
\end{equation}
Then applying Corollary \ref{cor-mon} to $u_{\delta_j^k}$ and
using the semicontinuity of $S$ Lemma \ref{lem-semi} together with Lemma \ref{lem-vigenik} we have
\begin{eqnarray}\label{balag-2}
2M\vol(B_1\cap \po{u_{00}^k})= \limsup_{y_k\to y_0} s(0, y_k, u_0)\le s(0, y_0, u_0)=2M\pi\Theta(x_0, \po{u_0}).
\end{eqnarray}

Therefore we conclude that $\Theta(x, \{u_0>0\})\ge \frac12$ for every
free boundary point $x$ in some neighbourhood of $x_0$. By virtue of  Lemma \ref{lem-u-0-visc} $u_0$ is a viscosity 
solution which is flat $x_0$. Applying the "flatness-implies-$C^{1, \alpha}$ "regularity results from \cite{C-Harnack-0} and\cite{C-Harnack-x} the lemma  follows. 
\end{proof}

\smallskip
Next we prove a representation formula for $\lap u_0$.
\begin{lem}\label{lem-rep}
Let $u_0$ be as in Lemma \ref{lem-cap}. Then
\begin{itemize}
\item[(i)] $\H^2(\Gamma_{\frac12}\cap B_R)<\infty$, for any $R>0$,
\item[(ii)] away from $\Gamma_1$ the representation formula holds
\[\lap u_0=\sqrt{2M}\H^2\with\Gamma_{\frac12}.\]
\end{itemize}
\end{lem}
\begin{proof}
(i) \ For given $x\in \Gamma_{\frac12}$ there is a $\widetilde \rho_x>0$ such that
\begin{equation}\label{blya1}
\sup_{B_r(x)}u_0\ge \sqrt{M}r, \quad r\in (0, \widetilde\rho_x).
\end{equation}
This follows from the asymptotic expansion in Lemma \ref{lem-asym}.
Consequently,  there is $\rho_x'>0$ such that
\begin{equation}\label{blya2}
\int_{B_r(x)}\lap u_0\ge \sqrt{M}r^2, \quad  r\in (0, \rho_x').
\end{equation}
Indeed, if this inequality is false then
there is a sequence $r_j\searrow 0$ such that
$$\int_{B_{r_j}(x)}\lap u_0<\sqrt{M}r_j^2.$$ Set $v_j(x)=\frac{u_0(x+r_jx)}{r_j}.$
By \eqref{blya1} $\sup_{B_1}v_j(x)\ge \sqrt{M}$. Moreover, it follows from
Lemma \ref{lem-asym} that $v_j(x)\to \sqrt{2M}x_1^+$ in a suitable coordinate system, 
while $\int_{B_1}\lap v\le \sqrt{M}$. However, $\int_{B_1}\lap x_1^+=\sqrt{2M}\frac\pi2$
and this is in contradiction with the former inequality. Putting $\bar \rho_x=\min(\rho_x', \widetilde \rho_x)$
we see  that  the  collection of  balls   $\mathcal F=\bigcup B_{\rho_x}(x), x\in \Gamma_{\frac12}\cap B_R, \rho_x<\bar \rho_x$
is a Besicovitch type covering of $\Gamma_{\frac12}\cap B_R$.
Consequently, there is a positive integer $m>0$ and subcoverings $\mathcal F_1, \dots, \mathcal F_m$
such that the balls in each of $\mathcal F_i, 1\le i\le m$ are disjoint and
$\Gamma_{\frac12}\cap B_R\subset \bigcup_{i=1}^m\mathcal F_i$.
We have from \eqref{blya2}
\[4\pi R^2\|\nabla u_0\|_\infty\ge \int_{\partial B_R}\p_\nu u_0\ge \int_{B_{\rho_x}(x)\in \mathcal F_i}\lap u_0=
\sum_{B_{\rho_x}(x)\in \mathcal F_i}\int_{B_{\rho_x}(x)}\lap u_0\ge m\sqrt{M}\sum_{B_{\rho_x}(x)\in \mathcal F_i}\rho_x^2.\]
This yields
\begin{equation}\label{blya3}
\sum_{B_{\rho_x}(x)\in \bigcup_{i=1}^m\mathcal F_i}\rho_x^2\le \frac{4\|\na u_0\|_\infty \pi R^2}{m\sqrt M}.
\end{equation}
Given $\delta>0$ small, suppose there is $x\in \Gamma_{\frac12}$ such that $\bar \rho_x\ge \delta$.
Then we choose $\rho_x<\delta.$ Thus, in any case we can assume that $\rho_x<\delta$.
In view of   \eqref{blya3} this implies that the $\delta$-Hausdorff premeasure is bounded independently of
$\delta$. This proves (i).

(ii) From the estimate
\[\sqrt{M}r^2\le \int_{B_r(x)}\lap u_0\le 4\pi r^2\|\na u_0\|, \quad r\in (0, \bar \rho_x), B_r(x)\cap \Gamma_{\frac12}\subset \Gamma_{\frac12}\]
we see that there is a positive bounded function $q$ such that
$\lap u_0=q\H^2\with \Gamma_{\frac12}$. Using Lemma \ref{lem-asym}
we conclude that $q=\sqrt{2M}$.
\end{proof}

Next we prove the full non-degeneracy of $u_0$ near $\Gamma_{\frac12}$.

\begin{lem}\label{lem-ndg}
Let $u_0$ be as above and $x_0\in \Gamma_{\frac12}$ then for any $B_r(x)$
such that $x\in \fb{u_0}, B_r(x)\cap \fb{u_0}\subset \Gamma_{\frac12}$ we have
\[\sup_{B_r(x)}u_0\ge \sqrt{2M}\pi r.\]
\end{lem}

\begin{proof}
By a direct computation we have
\[r^{-2}\int_{\p B_r(x)}u_0=\int_0^r\frac{dt}{t^2}\int_{B_t(x)}\lap u_0\ge\int_0^r\frac1{t^2}\sqrt{2M}\pi t^2=\sqrt{2M}\pi r\]
where the inequality follows from the representation formula and the fact that
$\fb{u_0}$ is a cone, and hence for all $t\in (0, r)$ $\H^2(B_r(x)\cap \Gamma_{\frac12})\ge \pi t^2$.
It remains to note that $r^{-2}\int_{\p B_r(x)}u_0\le \sup_{B_r(x)}u_0$.
\end{proof}

\subsection{Weak solutions}

Combining Lemmas \ref{lem-rep} and \ref{lem-ndg} as well as Propositions \ref{prop-compactness} (iii) and
\ref{prop-tech} (i) we see that $u_0$ is a weak solution near $\Gamma_{\frac12}$ in the sense of Definition 5.1 \cite{AC}.
Furthermore, $\fb{u_0}\setminus \{0\}$ is flat at each point.

\begin{lem}\label{lem-AC}
$u_0$ is a weak solution in Alt-Caffarelli sense away from $\Gamma_1$. Furthermore,  $\Gamma_{\frac12}$ is smooth.
\end{lem}

\begin{proof}
All conditions in Definition 5.1 \cite{AC} are satisfied and $u_0$ is flat at every point $z_0\in \fb{u_0}\setminus\{0\}$
thanks to \eqref{asym-exp}. Applying Theorem 8.1 \cite{AC} we infer that $\Gamma_{\frac12}$ is smooth  at every $z_0\in \fb{u_0}\setminus\{0\}$.
\end{proof}

\subsection{Minimal perimeter}
In this section
we prove that the local perturbations $S'\subset\po{ u_0}$ of
a portion $S\subset \Gamma_{\frac12}$ has larger
$\m H^{2}$ measure   than $S$.
This can be seen from the estimate $|\nabla u_0(x)|\le \sqrt{2M}$
which follows from Lemma \ref{lem-limsup}.
Since by Lemma \ref{lem-AC} on $\Gamma_{\frac12}$ the free boundary
condition $|\na u_0|=\sqrt{2M}$ is satisfied in the classical sense, it follows that
\begin{eqnarray*}
0=\int_{D}\lap u_0=\int_{S}\p_\nu u_0+\int_{S'}\p_\nu u_0=\\
=\sqrt{2M}\H^2(S)+\int_{S'}\p_\nu u_0,
\end{eqnarray*}
where $D\subset \po{u_0}$ such that $\p D=S\cup S'$.
But $\left|\int_{S'}\p_\nu u_0\right|\le \sqrt{2M}\H^2(S')$ and thereby
\begin{equation}\label{perim}
\H^2(S)\leq \H^2(S').
\end{equation}

\medskip

The estimate for the perimeter can be reformulated as follows:
\begin{theorem}\label{thrm-convex}
 Let $N=3$, then the components of $\Gamma_{\frac12}$ are surfaces of non-positive outward mean curvature.
 In particular, $\Gamma_{\frac12}$ is a union of smooth convex surfaces.
\end{theorem}
\begin{proof}
Since $u_0$ is a weak solution then by Lemma \ref{lem-AC} $\Gamma_{\frac12}$ is smooth. If $z_0\in \Gamma_{\frac12}$
then choosing the coordinate system in $\R^3$ so that $x_3$-axis  has the direction of the inward normal of $\po{u_0}$
at $z_0$ and considering the free boundary near $z_0$ as a graph $x_3=h(x_1, x_2)$
we can consider  the one-sided variations of the surface area functional.
Indeed, let $\mathcal D\subset \R^2$ be a open bounded domain in $x_1x_2$ plane containing
$z_0$ and assume $t>0, 0\le \psi\in C_0^\infty(\mathcal D)$. Then from  \eqref{perim} we have
\begin{eqnarray}
0&\ge& \frac1t\int_{\mathcal D}[\sqrt{1+|\na h|^2}-\sqrt{1+|\na(h-t\psi)|^2}]=\\\nonumber
&=&\int_{\mathcal D}\frac{2\na h\na\psi-t|\na \psi|^2}{\sqrt{1+|\na h|^2}+\sqrt{1+|\na(h-t\psi)|^2}}\to\ \ \hbox{as}\ t\to 0 \\\nonumber
&\to&\int_{\mathcal D}\frac{\na h\na\psi}{\sqrt{1+|\na h|^2}}.
\end{eqnarray}
Therefore $\hbox{div}\left(\frac{\na h}{\sqrt{1+|\na h|^2}}\right)\ge 0$ and noting that
$\Gamma_{\frac12}$ is a cone the result follows.
\end{proof}

\subsection{Full Non-degeneracy}
\begin{lem}\label{lem-deg}
Assume that  $N=3$ and let $u_0$ be a nontrivial blow-up of $u$
such that the measure theoretic boundary of $\po{u_0}$ is non-empty. 
Then $\fb{u_0}\setminus\{0\}\subset \Gamma_{\frac12}$. 
In particular  the set of degenerate points of $\fb{u_0}$ is empty.
\end{lem}

\begin{proof}
Let $u_0$ be a blow-up of $u$ at 0. Since $u$ is non-degenerate at
0 then it follows that $u_0$ does not vanish identically. Hence there is
a ball $B\subset \po{u_0}$ touching  $\fb {u_0}$ at some point $z_0\in \fb {u_0} \cap B$. By Hopf's lemma,
Lipschitz estimate \ref{prop-tech} (i) and asymptotic expansion \cite{C-Harnack-x}  Lemma A1 it follows
that $u_0$ is not degenerate at $z_0$.
Consequently, the set of non-degenerate points of $u_0$ is not empty.

Suppose that $S$ is a component of $\fb{u_0}$ containing a point of measure theoretic boundary of $\po {u_0}$. Note that by Lemma 
\ref{lem-C1alpha} and Theorem \ref{thrm-convex} $S$ is a smooth convex surface. Let $x_0\in\partial S, x_0\not =0$. 
Then either 
a) $x_0\in \Gamma_{1}$ or b) $u_0$ is degenerate at $x_0$.

We first analyze the case a). Let $\ell$ be the ray passing through 
$x_0$ and $\Pi$ the tangent half-plane to $S$ along $\ell$.
First note that $u_0$ is non-degenerate at $x_0$ because 
\[
\int_{B_r(x_0)}\Delta u_0\ge \int_{B_r(x_0)\cap S}\Delta u_0\ge \sqrt{2M}\H^2(S\cap B_r(x_0))\ge \sqrt{2M}\frac{\pi r^2}2
\]
for sufficiently small $r$. 
Consequently 
\begin{equation}\label{holms}
\frac1{R^2}\int_{\p B_R} u_0=\int_0^R\frac1{r^2}\int_{B_r(x_0)}\Delta u_0\ge \sqrt{2M}\frac{\pi r^2}{2}R. 
\end{equation}
Let $u_{00}$ be a blow-up $u_0$ at $x_0$. Then from Theorem A it follows that 
$u_{00}$ is two dimensional. Moreover $\Pi\subset \fb{u_{00}}$, $\po{u_{00}}$ has unit density at $0$,  and 
the interior of $\{u_{00}=0\}$ near $\Pi$ is not empty. 
Note that the interior of the set $\{u_0=0\}$ propagates to $x_0$ along 
another component $S_1$ of measure theoretic boundary, see Figure \ref{fig-3}. 
Consequently,  near $\Pi$ $u_{00}(x)=\sqrt{2M}x_1^+$ after some rotation of coordinates.
From the unique continuation theorem it follows that $u_{00}(x)=\sqrt{2M}x_1^+$ everywhere which is in contradiction with the fact that 
$\po{u_{00}}$ has unit density at $0$.  

As for the case $b)$ then \eqref{holms} shows that $u_0$
is non-degenerate at $x_0$ as long as $x_0$ is on the boundary of $S$.

\begin{figure}
\begin{center}
 \begin{tikzpicture}

    \coordinate (O) at (3,-5.25) ;
    \coordinate (A) at (0,0) ; 
        \coordinate (A1) at (-0.6,+1.05) ; 
    \coordinate (B) at (2.7,1.65) ; 
        \coordinate (C) at (-1.5,-2.25) ; 

 \draw[ultra thick] (O) node[below right] {$\resizebox{0.017\hsize}{!}{$O$}$}--(A1) ;  
         \node[below left] at (A)  {{$x_0$}};
                  \node[below left] at (A1)  {$\resizebox{0.01\hsize}{!}{$\ell$}$};

        \filldraw[cyan,opacity=.2] (A) to[out=10,in=237] (B)--(O); 
        \filldraw[cyan,opacity=.2] (A) to[out=-70,in=30] (C)--(O); 

              \draw[->, blue, ultra thick ] (A)-- (10:2) ;
      \filldraw[blue!60!,opacity=.3] (A)--(3.4,0.6)--(O);

        \node[below right] at (-1.2,-2.75) {{$S_1$}};
                \node[below right] at (2,1.75) {{$S$}};
                                \node[below right] at (3.2,-1.75) {{$\Pi$}};

\end{tikzpicture}
\end{center}
\caption{The structure of free boundary near the point $x_0$. }
\label{fig-3}
\end{figure}
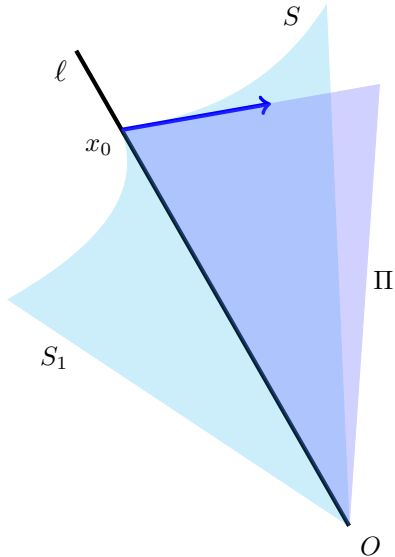

\end{proof}

\subsection{Properties of $\Gamma_{\frac12}$}

\begin{lem}
Suppose $u_0$ is not degenerate at $x_0\in \fb{u_0}\setminus\{0\}$, such that
$\Theta_*(x_0, \{u_0=0\})>0$. Then there is a unique component $\mathfrak C$ of $\fb{u_0}$
containing $x_0$ such that $\mathfrak C\subset \Gamma_{\frac12}.$
\end{lem}

\begin{proof}
We only have to show the uniqueness of $\mathfrak C$, the rest follows from Lemmata \ref{lem-asym} and \ref{lem-C1alpha}.
Suppose, there are two components of $\fb{u_0}\setminus\{0\}$,  $\mathfrak C_1$ and $\mathfrak C_2$,
containing $x_0$. From the dimension reduction argument as in the proof of Lemma \ref{lem-asym}
it follows that $\mathfrak C_1$ and $\mathfrak C_2$ have the same approximate tangent plane at
$x_0$. This is in contradiction with our assumption $\Theta_*(x_0, \{u_0=0\})>0$.

\end{proof}

\begin{lem}\label{lem-cmpt}
Let $\mathfrak C$ be a component of $\fb{u_0}$ such that $\mathfrak C\cap \Gamma_{\frac12}\not =\emptyset$.
Then $\mathfrak C\setminus \Gamma_{\frac12}=\emptyset$, in other words all points of
$\mathfrak C$ are in $\Gamma_{\frac12}$.
\end{lem}
\begin{proof}
By Lemma \ref{lem-deg} $\mathfrak C$ cannot have degenerate points, thus we have to show that
$\Gamma_{\frac12}$ cannot have limit points in $\Gamma_1$.
Note that $\Gamma_{\frac12}$ is of locally finite perimeter (see Lemma \ref{lem-rep} (i)) and hence locally it is
a countable union of convex surfaces. Let $x_0\in \Gamma_{1}\cap \mathfrak C$ be a limit point of
$\Gamma_{\frac12}\cap \mathfrak C$. The generatrix of the cone $\fb{u_0}$ passing through $x_0$ splits $\mathfrak C$ into
two parts one of which must be convex near $x_0$ because by assumption $x_0$ is a limit point of $\Gamma_{\frac12}$,
see Theorem \ref{thrm-convex}. The set $\{u_0=0\}^\circ$ propagates to $x_0$ because
$\Gamma_{\frac12}$ is a subset of reduced boundary. Thus,  there is another
subset of $\Gamma_{\frac12}$ approaching to $x_0$, and it is a part of the topological boundary of
$\{u_0=0\}^\circ$. Therefore,  the ray passing through $x_0$ is on the boundaries of two
convex pieces  of $\fb{u_0}$ (near $x_0$). Note that if these pieces of
$\Gamma_{\frac12}$ contain flat parts then from the unique
continuation theorem we infer that $\fb{u_0}$ cannot have singularity at $0$.
Thus, they cannot contain flat parts and consequently the density of $\po{u_0}$ at $x_0$
cannot be 1, because
by convexity of $\Gamma_{\frac12}$ it follows that $\{u_0\equiv0\}^\circ$ has
positive density at $x_0$. But  this is in contradiction with the assumption $x_0\in \Gamma_1$.
\end{proof}

Summarizing  we have
\begin{prop}\label{prop-sum}
Let $u_0$ be as above and $N=3$, then $\fb{u_0}\setminus\{0\}$
is a union of smooth convex cones.
\end{prop}

\subsection{Proof of Theorem B}
The first part of Theorem B follows from Lemma \ref{lem-deg} while the second part
is a corollary of Lemma \ref{lem-cmpt} since  $\Gamma_{\frac12}$ coincides with the reduced boundary. Finally, the last part follows from Lemma \ref{lem-AC}, because by Lemma \ref{lem-cmpt}
the reduced boundary propagates instantaneously in  $\fb{u_0}$.

\section{Proof of Theorem C}
\subsection{Inverse Gauss map and the support function}
Suppose $ N=3$ and $u=rg(\theta, \phi),$ where 
\[x=r(\sin\theta\cos\phi, \sin\theta\sin\phi, \cos\theta)\]
then 
\[\lap u_0=\frac1r\left(g_{\theta\theta}+\frac{\cos\theta}{\sin\theta}g_\theta+\frac{g_{\phi\phi}}{\sin^2\theta}+2g\right).\]

Note that 
\[\lap_{\S^2}g=g_{\theta\theta}+\frac{\cos\theta}{\sin\theta}g_\theta+\frac{g_{\phi\phi}}{\sin^2\theta}\]
is the Laplace-Beltrami operator.
Thus we get
\begin{equation}\label{mean0}
\lap_{\S^2}g+2g=0.
\end{equation}
Let $H(n), n\in \S^{N-1}$ be the Minkowski support function of some hypersurface 
$\M$.
$H(n)$ is the distance between  the point on $\M$ with normal $n$ 
and the origin. 
It is known \cite{Aleksandrov} that the eigenvalues of the matrix 
\[\na^2_{ij} H(n)+\delta_{ij}H(n)\]
are the principal radii of curvature of the surface determined by $H$, where  the second order derivatives are taken with respect to an orthonormal frame at $n\in \S^{N-1}$. The support function uses the inverse of the Gauss map
to parametrize the surface as follows
\[H(n)=G^{-1}(n)\cdot n.\]
Furthermore, we have the following formula for the Gauss curvature $K$  \cite{Aleksandrov}
\begin{equation}\label{Gauss-det}
\frac1K=\det(\na^2_{ij} H(n)+\delta_{ij}H(n)).
\end{equation}
The Gauss map is a local diffeomorphism  whenever $K\not=0$ \cite{Rosenberg}. Since
$u_0=rg$ is harmonic in $\po{u_0}$ we infer that $g$ is smooth on $\S^2\cap \po{g}$.

\subsection{Catenoid is a solution}
In \cite{AC}  Alt and Caffarelli constructed a weak solution which is not a minimizer.
Their solution can be given explicitly as follows: let
\[x=r(\sin\theta\cos\phi, \sin\theta\sin\phi, \cos\theta)\]
and take
\[u(x)=r\max\left(\frac{f(\theta)}{f'(\theta_0)}, 0\right)\]
where
\[f(\theta)=2+\cos\theta\log\left(\frac{1-\cos\theta}{1+\cos\theta}\right)=2+\cos\theta\log\left(\tan^2\frac\theta2\right)\]
and $\theta_0$ is the unique zero of $f$ between 0 and $\frac\pi2$.
The aim of this section is to show that $f$ is the support function of some catenoid.
Recall that the principal radii of curvature of a smooth surface are the
eigenvalues of the matrix $\D_{\S^{N-1}}^2H+\delta_{ij}H$ where
the Hessian is taken with respect to the sphere $\S^{N-1}$ \cite{Aleksandrov}. At each point
where the Gauss curvature does not vanish the zero mean curvature condition for $N=3$ can be written as
\[\triangle_{\S^2}H+2H=0\]
where $\triangle_{\S^{2}}$ is the Laplace-Beltrami operator and
$H(n)$ is the value of Minkowski's support function corresponding to the normal $n\in \S^2$.
Recall that, in $x, y$ coordinates, 
by rotating the graph of $y(x)=a\cosh \frac xa, a=const.$ around the $x-$axis one obtains a catenoid. 
Thus it is enough to compute the support function
for the graph of $y$. Let $\alpha$ be the angle  that the
tangent line of $y$ at $(x, y(x))$ forms with the $x-$axis. If $n$ is the unit normal to the graph of
$y$ then $n=(-\sin\alpha, \cos\alpha)$ and
\[H(n)=(x, y(x))\cdot n=-x\sin\alpha+a\cos\alpha \cosh\frac xa.\]
Noting that the unit tangent at $(x,y(x))$ is $(\cos\alpha, \sin \alpha)$ and equating with 
the slope of tangent line, which is $(\sinh \frac xa, -1)$,
we obtain
\[\cos\alpha=\frac{\sinh\frac xa}{\sqrt{1+\sinh^2\frac xa}}, \quad \sin\alpha =-\frac{1}{\sqrt{1+\sinh^2\frac xa}}.\]

From the second equation we get that $\sinh\frac xa=\tan\alpha$ and solving the
quadratic equation $e^{2\frac xa}-1=2e^{\frac xa}\tan\alpha$ we find that
\[x=a\log\frac{1+\sin\alpha}{\cos\alpha}, \quad \cosh\frac xa=\frac1{\cos\alpha}.\]
Consequently,
\[H(n)=-\frac a2\sin\alpha\log\left(\frac{1+\sin\alpha}{\cos\alpha}\right)^2+a.\]
Taking $\alpha=\theta+\frac\pi2$ we have
\[\frac{1+\sin\alpha}{\cos\alpha}=\frac{1+\cos\theta}{-\sin\theta}=\frac{2\cos^2\frac\theta2}{-2\sin\frac\theta2\cos\frac\theta 2}=-\cot\frac\theta2\]
and the result follows if we choose $a=2$.


\subsection{Almost minimal immersions}\label{subsec-immersion}
Consider the parametrization $\X:U_g\to \R^3$, where
\begin{equation}\label{X-par}
\X(n)=ng(n)+\nabla_{\S^2}g, \quad U_g=\po g\subset \S^2.	
\end{equation}
Let $\M$ be the hypersurface determined by $\X$.
The spherical part $g$ of $u_0$ solves the equation \eqref{mean0}
and by Theorem 1 \cite{Rez} $\X$ determines a smooth map which is either constant
or a conformal minimal immersion outside locally finite set of isolated singularities (branch points).
Recall that
if at some point $p$
\begin{equation}
\X_{\xi_1}\times\X_{\xi_2}=0, \X=\X(\xi_1, \xi_2 )\ \text{in local coordinates}\ \xi_1, \xi_2, 
\end{equation}
then  $p$ is  called branch point, see  \cite{N-book} page 314.

Observe that $\X(n)$ is the gradient of the blow-up $u_0$ at $n=\frac{x}{|x|}$. Indeed,
\begin{eqnarray}\label{X-grad}
\X(n)&=&\frac n{r}rg+\frac1r\nabla_{\S^2}(rg)=\\\nonumber
&=&\frac n{r}u_0(x)+\frac1r\nabla_{\S^2}u_0=\\\nonumber
&=& \frac n{r}(\nabla u_0(x)\cdot x)+\frac1r\nabla_{\S^2}u_0=\\\nonumber
&=&n(\nabla u_0(x)\cdot\frac{x}{|x|})+\frac1r\nabla_{\S^2}u_0=\\\nonumber
&=&\nabla u_0(x).
\end{eqnarray}

In particular, the computation above shows that

\begin{equation}\label{grad-hom}
\nabla u_0(x)=\nabla u_0(\frac x{|x|}), \quad \nabla_{\S^2}g(n)\perp n,
\end{equation}
in other words the gradient is homogeneous of degree zero.

The absence of branch points does not rule out the possibility of self-intersection. Therefore we need to prove that under conditions of Theorem C $\M$ is embedded.

\subsection{Dual cones and center of mass}
If $u_0$ is a blow-up and the assumptions in  Theorem C  are satisfied, 
then by virtue of Proposition \ref{prop-sum}  the free boundary  $\fb{u_0}\setminus \{0\}$
is a union  of smooth convex cones $\C_1$ and $\C_2$.
We define the dual cones as follows
\begin{equation}\label{vigenik-dual}
\C^\star_i=\p \{y\in \R^3 : x\cdot y\le 0, x\in \C_i\}, \quad i=1,2.
\end{equation}
It is well-known that the dual of a convex cone is also convex, \cite {Schneider} page 35.

\begin{lem}\label{lem-pos-con}
The largest principal curvature of $\C_i\setminus\{0\}$ is strictly positive.
\end{lem}
\begin{proof}
To fix the ideas we prove the statement for $\C_1$.
Note that one of the principal curvatures of $\C_1\setminus\{0\}$ is zero
because $\C_1$ is a cone and $\C_1\setminus\{0\}$ is smooth, see Theorem B.
Let $\kappa(p)$ be the largest principal curvature at $p\in \C_1\setminus\{0\}$.
Suppose there is $p$ such that $\kappa(p)=0.$ 
Choose  the coordinate system at $p$
so that $x_1$ points in the outward normal direction at $p$ (into $\{u_0\equiv 0\}$),
$x_2$ axis  is tangential at $p$ and is the principal direction corresponding to $\kappa(p)$.
Then we have that $\nabla u_0(p)=e_1$, the unit direction of $x_1$ axis
and the mean curvature of $\C_1$ at $p$ vanishes because we assumed that $\kappa(p)=0.$
Writing the mean curvature at $p$ in terms of the derivatives of $u_0$ we have
\[0=\frac{\nabla u_0 D^2 u_0(\nabla u_0)^T-|\nabla u_0|^2\lap u_0}{|\nabla u_0|^3}=\frac{\p_{11}u_0}{\sqrt{2M}}\]
implying that $\p_{11}u_0=0$. Moreover, since $u_0$ is homogeneous of degree one then
$\nabla u_0=e_1$ along the $x_1$ axis. This yields $\p_{13}u_0=\p_{23}u_0=\p_{33}u_0=0$
along the $x_1$ axis. From the harmonicity of $u_0$ 
 it follows that  $\p_{22}u_0=0$
along the $x_3$ axis.
Summarizing, we have that along the points of the $x_3$ axis the Hessian of $u_0$
has the following form
\[\begin{pmatrix}
  0 & \p_{12} u_0 &0\\
  \p_{12}u_0 & 0 &0\\
  0 &0 &0
\end{pmatrix}.\]
Finally, letting $\sigma(t), t\in(-\delta, \delta)$ be the parametrization of the curve
along which the $x_1x_2$ plane intersects with
$\C_1$ and differentiating $|\nabla u_0(\sigma(t))|=1$ in $t$ we get that at $p$
one must have
\[0=e_1\begin{pmatrix}
  0 & \p_{12}u_0 &0\\
  \p_{12}u_0 & 0 &0\\
  0 &0 &0
\end{pmatrix}e_2=\p_{12}u_0(p).\]
Thus, the Hessian $D^2 u_0$ vanishes along the $x_1$ axis. The function $w=\sqrt{2M}-\p_1 u_0$ is harmonic in $\po {u_0}$ 
and $w\ge 0$ thanks to Lemma \ref{lem-limsup}. Moreover, $w(e_1)=0=\min w$. Since at $e_1$ the free boundary is regular then by  Hopf's lemma   $\p_1w=-\p_{11}u_0\not =0$. 
However $D^2 u_0(te_1)=0$ for every $t>0$ and hence 
$\p_{11}w(e_1)=0$ which is a contradiction.
\end{proof}

\begin{remark}\label{rem-curv}
  It follows from  Lemma \ref{lem-pos-con} and Theorem B that there are two
  positive constants $\kappa_0,\kappa_1$ such that
  \[
  0<\kappa_0\le \kappa(p)\le \kappa_1, \quad p\in (\fb{u_0}\setminus\{0\})\cap \p B_{\sqrt{2M}},
  \]
  where $\kappa(p)$ is the largest curvature of $\fb {u_0}$ at $p\in (\fb{u_0}\setminus\{0\})\cap \p B_{\sqrt{2M}}$.
\end{remark}

Let us put  $\gamma_i=\S^2\cap \C_i^\star$.

\begin{lem}\label{lem-cone}
Let  $\C_1^\star, \C_2^\star$ be the dual cones \eqref{vigenik-dual}. Then we have
\begin{itemize}
\item[(i)] $\p \M$ is differentiable and there are two positive constants $\kappa_0^\star, \kappa_1^\star$
such that the largest curvature $\kappa^\star(p)$ of $(\C_i^\star\setminus \{0\})\cap \S^2$
satisfies $\kappa_0^\star\le \kappa^\star(p)\le \kappa_1^\star$,
\item[(ii)] there is $\delta>0$ small such that every component $E_\delta $ of $\p B_{1-\delta}\cap \M$ defines a convex cone $K_\delta=\{\sigma t : \sigma\in E_\delta, t>0\}$,
    \item[(iii)] $\M$ is star-shaped with respect to the origin and hence embedded.
\end{itemize}

\end{lem}

\begin{proof}
Suppose that $\C_1^\star$ is not differentiable at some $z\not =0$.
Then $\C_i$ must have a flat piece. Indeed, if $n_1, n_2$ are the normals of two  supporting hyperplanes of $\C_i^\star$ at $z$ then the
unit vectors $n_t=\frac{tn_1+(1-t)n_2}{|tn_1+(1-t)n_2|},$
define a support function at $z$ for every  $t\in(0,1).$
Since the vectors $n_t$ lie on the same plane then $\C_1$ must have
a flat piece. The unique continuation theorem implies that the free boundary is a hyperplane
and cannot have singularities. Now the desired estimate follows from Remark
\ref{rem-curv} and the definition of dual cone. The first claim is proved.

Let $k_1, k_2$ be the principal curvatures of $\M$, then $k_1+k_2=0$ and the Gauss curvature
is $K=-k_1^2=-k_2^2$. Since $\M$ is a smooth  immersion then from \eqref{Gauss-det}  and the smoothness of
$\X=\nabla u_0$ in $U_g$ we see that $K\not=0$. Furthermore, there is a tame constant
$c_0>0$ such that $k_i^2\ge c_0, i=1, 2$ at every point of $\M$.
Thus by virtue of the part (i) $\M$ is fibred by $\p B_{1-\delta}$ for  $\delta>0$ small. We claim that
$|\X(n)|>0, n\in\overline{U_g}$. Clearly this is true if $n\in \p U_g$ where
$|\X(n)|=1$. Suppose there is $n\in U_g$ such that $\X(n)=0$. Since
$\X(n)=ng+\nabla_{\S^2}g$ it follows that $g(n)=0$, but this is impossible since $n\in\{g>0\}=U_g.$
From  $g(n)=\X(n)\cdot n>0, n\in U_g$ it follows that $\M$ is starshaped with respect to the origin. Consequently, $\M$ is fibered by $\p B_t$ for every $t\in(0, 1)$ and hence embedded.
\end{proof}

Let $n\in U_g$ then  from $\X(n)=\na u_0(n)$ it follows that
\[|\X(n)|_{\fb{u_0}}|=
|\nabla u_0 |_{\fb{u_0}}|=\sqrt{2M}.\]
Since by  Lemma \ref{lem-cone} $\M$ is differentiable along $\gamma_i$ we see that the contact angle $\alpha$ between $\M$ and $\S^2$ is
\[\cos \alpha= n \cdot \frac{\X(n)|_{\fb{u_0}}}{\sqrt{2M}}=g(n)|_{\fb{u_0}}=0.\]
Thus, the minimal surface defined by $g$ is inside of the sphere of radius $\sqrt{2M}$ because 
in view of Lemma \ref{lem-limsup} $|\nabla u_0|^2=g^2+|\na g|^2\le 2M$.
Moreover,
$\M$ is tangential to $\C_1^\star$ and $\C_2^\star$ along $\S^2$ since $n\perp \nabla_{\S^2}g$ by \eqref{grad-hom}.

\medskip
We recall the definition  of topological type $[\e, r, \chi]$ 
of hypersurface $\M\subset \R^3$ from \cite{Nitsche} page 47.
\begin{defn}
We say that $\M$ is of topological type $[\epsilon, r, \chi]$ if it has orientation $\e$, Euler characteristic
$\chi$, and  $r$ boundary curves. Here $\epsilon=\pm1$, where $+1$ means that $\M$ is orientable
and $\epsilon=-1$ non-orientable. For orientable surface the Euler characteristic
is defined by the relation $\chi=2-2g-r$ where $g$ is the genus of $\M$.
\end{defn}

 Now the first part of Theorem C follows from Nitsche's theorem, see page 2 \cite{Nitsche}.
Moreover, the only stationary surfaces of  disk type are the totally geodesic
disks and the spherical cups. From  Lemma \ref{lem-cap} it follows that
if $u_0=rg$ and $\supp g$ is  a disk then $u_0$ is a half plane.

In view of Lemma \ref{lem-cone} (iii) the proof of Theorem C can be deduced from the result of Nitsche \cite{Nit-caten} but 
we will sketch a shorter proof based on Aleksandrov's  moving plane method and 
Serrin's boundary lemma.  We reformulate  Theorem C as follows
\begin{lem}
Let $\M$ be of topological type $[1, 2, 0]$, i.e. a ring-type minimal surface.
Then $\M$ is a part of catenoid.
\end{lem}
\begin{proof}
By Lemma  \ref{lem-cone} (iii) $\M$ is embedded. In particular, 
$\X$ is a conformal minimal immersion (see the discussion in Section \ref{subsec-immersion}).

Let $\p\M=\gamma_1\cup\gamma_2$. Then applying Stokes formula
we have
\begin{equation}\label{mass-eq-1}
\int_{\M}\lap_\M \X=\int_{\p\M}n^*ds=\int_{\gamma_1}n^*+\int_{\gamma_2}n^*ds
\end{equation}
where $n^*$ is the outward conormal, i.e. $n^*$ is tangent to $\M$ and normal to
$\p\M$, see \cite{Fang} page 81. Since $\X$ is minimal then $\lap_\M \X=0$.
Thus
\begin{equation}\label{mass}
\int_{\gamma_1}n^*ds+\int_{\gamma_2}n^*ds=0.
\end{equation}

Since $\M$ is tangential to $\C_i^\star$ it follows that
the conormal $n^*$ on $\gamma_i$ points in the direction of the generatrix of the dual cone $\C_i^\star$.
Observe that if we use the arc-length  parametrization of 
$\gamma_i$ and let $s_k\in[0, |\gamma_i|]$ be some partition points
 then the sums 
$S_m=\sum_{k=0}^m n^{*i}_k(s_{k+1}-s_k), n_k^{*i}\in \C_i^\star$ approximate the
boundary integrals in \eqref{mass-eq-1}. Consequently  the vector $S_m$ is strictly inside of the cone $\C_i^\star$ and in the limit converges to
the centre of mass  of $\gamma_i$ computed with respect of the origin (the vertex of the cone).
In view of \eqref{mass} there is a diameter of $\S^2$ strictly inside of both dual cones
$\C_1^\star$ and $\C_2^\star$.

Without loss of generality we assume that the diameter passes through the north and south poles.
Now we can apply Aleksandrov's moving plane method and Serrin's boundary point
lemma to finish the proof. Let $\Pi_t$ be the family of planes containing $x_1$ axis and
$t$ measures the angle between $\Pi_t$ and $x_3$ axis.

Now start rotating $\Pi_t$ about $x_1$ axis starting form a position when
$\Pi_t$ is a support hyperplane to either of the cones $\C_1^\star, \C_2^\star$
and $\Pi_t\cap \C_i^\star\not=\emptyset, i=1,2$. 

{\bf Case 1:}
If the first touch of $\M$ and its reflection $\widetilde\M$ with respect to the
plane $\Pi_t$ occurs at some interior point of $\M$.
Then from the maximum principle it follows that $\M=\widetilde \M$.

By Lemma \ref{lem-cone}, both dual cones are strictly convex.
Moreover, we claim that for $\delta$ small
the cones generated by $\M\cap \p B_{1-\delta}$ are convex, otherwise
the inflection point would propagate to $\C_i^\star$.

The two remaining possibilities are:

{\bf Case 2:} if the first touch of $\M$ and its reflection $\widetilde\M$
occurs  at some boundary point where $\p\M$  is perpendicular to $\Pi_t$,

{\bf Case 3:} if the first touch of $\M$ and its reflection $\widetilde\M$
occurs  at some  boundary point where $\p\M$ is not lying on $\Pi_t$.

We cannot directly apply Serrin's boundary point lemma \cite{Serrin}
because $\p\M$ is only $C^{1,1}$ by virtue of Lemma \ref{lem-cone}. However, from the fibering
of $\M$ near $\p\M$ we conclude that $\widetilde g\le g$ near the contact point, where
$\widetilde g$ is the support function of $\widetilde M$.
Thus $\widetilde u=r\widetilde g\le rg=u$. Hence applying
Serrin's boundary point lemma to the harmonic functions $\widetilde u$ and $u$
we conclude that
$\M=\widetilde \M$.

Choosing $\Pi_t$ to be an arbitrary family passing through a line perpendicular to the
diameter it follows that
 $\gamma_1, \gamma_2$ are circles and \eqref{mass} forces them to lie
on parallel planes. Applying  Corollary 2 \cite{Schoen} we infer that $\M$ is a part of catenoid.
\end{proof}

\section{Appendix}
This section contains some well known results about  the solutions of the singular perturbation problem \eqref{pde-0}.
We begin with the uniform Lipschitz estimates of Luis Caffarelli, see  \cite{C-95} for the proof.
\begin{prop}\label{prop1}
Let $\{u_\e\}$ be a family of solution of \eqref{pde-0} then
there is a constant $C$ depending only on $N, \|\beta\|_{\infty}$ and independent of $\e$ such that
\begin{equation}\label{grad-est}
\|\na u_\e\|_{L^\infty(B_{\frac12})}\leq C.
\end{equation}
\end{prop}

As a consequence we get that one can extract converging sequences $\{u_{\e_n}\}$ of solutions of \eqref{pde-0}
such that  the limit functions are stationary points of the Alt-Caffarelli problem.

\begin{prop}\label{prop-compactness}
Let $u_\e$ be a family of solutions to \eqref{pde-0} in a
domain $\mathcal D\subset \R^N$. Let us assume that
$\|u_\e\|_{L^\infty(\mathcal D)}\le A$ for some constant
$A>0$ independent of $\e.$ For every $\e_n\to 0$ there
exists a subsequence $\e_{n'}\to 0$ and $u\in C^{0, 1}_{loc}(\mathcal D)$, such that
\begin{itemize}
\item[(i)] $u_{\e_{n'}}\to u$ uniformly on compact subsets of $\mathcal D$,
\item[(ii)] $\na u_{\e_{n'}}\to \na u$ in $L^2_{loc}(\mathcal D)$,
\item[(iii)] $u$ is harmonic in $\mathcal D\setminus \fb u$.
\end{itemize}

\end{prop}
\begin{proof}
See Lemma 3.1 \cite{CLW-uniform}.
\end{proof}

Next, we recall the estimates for the slopes of some global solutions.
\begin{prop}\label{prop-tech}
Let $u$ be as in Proposition \ref{prop-compactness}. Then the following statements hold true:
\begin{itemize}
\item[(i)] $u$ is Lipschitz,
\item[(ii)] if $u_{\e_j}\to u=\alpha x_1^+$ locally uniformly, then $0\le 
\alpha \le \sqrt{2M}$, see Proposition 5.2 \cite{CLW-uniform},
\item[(iii)] if $u_{\e_j}\to u=\alpha x_1^+-\gamma x_1^-+o(|x|)$ and $\gamma>0$ then
 $\alpha^2-\gamma^2=\sqrt{2M}$,  see Proposition 5.1 \cite{CLW-uniform}.
In this lemma the essential assumption is that $\gamma>0$.
\end{itemize}
\end{prop}
\medskip
\begin{remark}\label{rem-wedge}
Observe that if $u(x)=\alpha x_1^++\bar\alpha x_1^-$  then we must necessarily have that $\alpha=\bar\alpha\le \sqrt{2M}$, see
Proposition 5.3 \cite{CLW-uniform}.
In this case the interior of the zero set of $u$ is empty. Thus one might have wedge-like solution.
\end{remark}


\medskip
Using  Proposition \ref{prop1} we can extract a sequence  $u_{\e_j}$ for some sequence $\e_j$
such that $u_{\e_j}\to u$ uniformly  in $B_{\frac12}$, see Proposition \ref{prop-compactness}.
Let $u$ be a limit  and $0<\rho_j\downarrow 0$ and $u_j(x)=\frac{u(x_0+\rho_j x)}{\rho_j}, x_0\in \fb u$. Thanks to Proposition \ref{prop-tech}(i)
we can extract a subsequence, still labeled $\rho_j$, such that $u_j$ converges to some function
$u_0$ defined in $\R^N$. The function $u_0$ is called a blow-up limit of $u$ at the free boundary
point $x_0$ and it depends on $\{\rho_j\}$.

The two propositions to follow establish an important property of the blow-up limits, namely that
the first and second blow-ups of $u$ can be obtained from \eqref{pde-0} for a suitable choice of  parameter
$\e$. Observe that the scaled function $\na (u_{\e_j})_{\lambda_n}$ verifies the equation
\begin{equation}\label{eq-lap-scale}
\lap  (u_{\e_j})_{\lambda_j}=\frac{\lambda_j}{\e_j}\beta\left(\frac{\lambda_j}{\e_j}(u_{\e_j})_{\lambda_j}\right).
\end{equation}
Taking $\delta_j=\frac{\e_j}{\lambda_j}\to 0$ we see that
$ (u_{\e_j})_{\lambda_j}$ is  solution to $\lap u_{\delta_j}=\beta_{\delta_j}(u_{\delta_j})$.

\begin{prop}\label{prop-1st-blow}
Let $u_{\e_j}$ be a family of solutions to \eqref{pde-0}
in a domain $\mathcal D\subset \R^N$ such that
$u_{\e_j}\to u$ uniformly on $\mathcal D$ and $\e_j\to 0.$
Let $x_0\in \mathcal D\cap \fb u$ and let $x_n\in \fb u$ be such that
$x_n\to x_0$ as $n\to \infty$. Let $\lambda_n\to 0, u_{\lambda_n}(x)
=(1/\lambda_n)u(x_n+\lambda_nx)$ and $(u_{\e_j})_{\lambda_n}=(1/\lambda_n)u_{\e_j}(x_n+\lambda_n x)$. Assume that $u_{\lambda_n}\to U$ as $n\to \infty$ uniformly on compact subsets
of $\R^N$. Then there exists $j(n)\to \infty$ such that for every
$j_n\ge j(n)$ there holds that $\e_j/\lambda_n\to 0$ and
\begin{itemize}
\item $(u_{\e_{j_n}})_{\lambda_n}\to U$ uniformly on compact subsets of
$\R^N$,
\item $\na (u_{\e_{j_n}})_{\lambda_n}\to \na U$ in $L^2_{loc}(\R^N),$
\item $\na u_{\lambda_n}\to \na U$ in $L^2_{loc}(\R^N)$.
\end{itemize}
\end{prop}
\begin{proof}
See Lemma 3.2 \cite{CLW-uniform}.
\end{proof}

Finally, recall  that the result of previous proposition extends to the second blow-up.
\begin{prop}\label{prop-2nd-blow}
Let $u_{\e_j}$ be a solution to \eqref{pde-0} in a domain
$\mathcal D_j\subset \mathcal D_{j+1}$ and $\cup_j \mathcal D_j=\R^N$
such that $u_{\e_j}\to U$ uniformly on compact sets of $\R^N$  and
$\e_j\to 0$. Let us assume that for some choice of positive numbers $d_n$
and points $x_n\in \fb U$, the sequence
\[U_{d_n}(x)=\frac1{d_n}U(x_n+d_nx)\]
converges uniformly on compact sets of $\R^N$ to a function $U_0$.
Let
\[(u_{\e_j})_{d_n}=\frac1{d_n}u_{\e_j}(x_n+d_nx).\]
Then there exists $j(n)\to \infty$ such that for every $j_n\ge j(n),$
there holds $\e_{j_n}/d_n\to 0$ and
\begin{itemize}
\item $(u_{\e_{j_n}})_{d_n}\to U_0$ uniformly on compact subsets of
$\R^N$,
\item $\na (u_{\e_j})_{d_n}\to \na U_0$ in $L^2_{loc}(\R^N).$\end{itemize}
\end{prop}
\begin{proof}
See Lemma 3.3 \cite{CLW-uniform}.
\end{proof}

Next lemma contains one of the crucial estimates needed for the 
proof of Proposition \ref{prop-sum}.

\begin{lem}\label{lem-limsup}
  Let $u\ge 0$ be as in Proposition \ref{prop-compactness}. Then
  \[\limsup_{x\to x_0, u(x)>0}|\na u(x)|\le \sqrt{2M}.\]
\end{lem}
\begin{proof}
  To fix the ideas we let $x_0=0$ and $l=\limsup\limits_{x\to 0, u(x)>0}|\na u(x)|$.
  Suppose $l>0$, otherwise we are done.
  Choose a sequence $z_k\to 0$ such that $u(z_k)>0$ and $|\nabla u(z_k)|\to l.$
  Setting $\rho_k=|y_k-z_k|$, where  $y_k\in \fb u)$ is the nearest point to $z_k$
  on the free boundary and proceeding as in  the proof of  \cite{AC-quasi} Lemma 3.4
  we can conclude that the blow-up sequence $u_k(x)=\rho_k^{-1}u(z_k+\rho_k x)$
  has a limit $u_0$ (at least for a subsequence, thanks to Proposition \ref{prop1})
  such that $u_0(x)=lx_1, x_1>0$ in a suitable coordinate system.
  Moreover, by Proposition \ref{prop-1st-blow} it follows that
  $u_0$ is a limit of some $u_{\lambda_j}$ solving $\lap u_{\lambda_j}=\beta_{\lambda_j}(u_{\lambda_j})$ in $B_{r_j}, r_j\to \infty$.
  If there is a point $z\in \{x_1=0\}$ and $r>0$ such that $u_0>0$ in $B_r(z)\cap \{x_1<0\}$
  then near $z$ we must have $u_0(x)=l(x-z)_1^++l(x-z)_1^-+o(x-z)$, see Remark \ref{rem-wedge}.
  Applying the unique continuation theorem to $u_0(x)-u_0(-x_1, x_2, \dots, x_n)$
  we see that $u_0=l(-x_1)^+, x_1<0$. Thus recalling Remark  \ref{rem-wedge} again we infer that
  $l\le \sqrt{2M}$.
\end{proof}

Finally, we mention a useful identity for the solutions $u_\e$, see equation (5.2) \cite{CLW-uniform}:
Let $u_\e$ be a solution of \eqref{pde-0} then for any $\phi\in C_0^\infty(B_1)$ there holds
\begin{equation}\label{domain-var}
\int\left(\frac{|\na u_\e|^2}2+\mathcal B(u_\e/\e)\right)\p_1\phi=\int\sum_k\p_ku_\e \p_1u_\e\p_k\phi.
\end{equation}

%

\end{document}